\documentclass[11pt]{article}

\usepackage[utf8]{inputenc}
\usepackage{amsmath}
\usepackage{amssymb}
\usepackage{amsthm}
\usepackage{enumitem}
\usepackage[colorlinks=true]{hyperref}
\usepackage[a4paper, top=3cm, bottom=2cm, left=3cm, right=2cm]{geometry}
\usepackage{multirow}
\usepackage{graphicx}
\usepackage{tikz}
\usepackage{arydshln}
\usepackage{subcaption}
\usepackage{authblk}
\usepackage{lineno}

\title{A framework for rigorous computational methods using Haar wavelets for differential equations}

\author{
    Guilherme Nakassima\thanks{Corresponding author} \textsuperscript{,}\hspace{1pt}\thanks{Instituto de Ciências Matemáticas e de Computação, Universidade de São Paulo. 400 Trabalhador São-Carlense Avenue, São Carlos, 13566-590, São Paulo, Brazil. E-mail: \texttt{gknakassima@gmail.com}} ,
    Marcio Gameiro\thanks{Department of Mathematics, Rutgers, The State University of New Jersey, 57 US Highway 1, Piscataway, 08854, NJ, USA. E-mail: \texttt{marciogameiro@gmail.com}}
}

\date{}

\DeclareMathOperator\supp{supp}

\newcommand{\R}{\mathbb{R}}

\theoremstyle{plain}
\newtheorem{theorem}{Theorem}[section]
\newtheorem{proposition}{Proposition}[section]

\newtheorem{lemma}{Lemma}[section]
\newtheorem{definition}{Definition}[section]

\theoremstyle{definition}
\newtheorem{remark}{Remark}[section]

\newtheorem*{notation}{Notation}

\begin{document}

\maketitle

\begin{abstract}
    This work presents a framework for a-posteriori error-estimating algorithms for differential equations which combines the radii polynomial approach with Haar wavelets. By using Haar wavelets, we obtain recursive structures for the matrix representations of the differential operators and quadratic nonlinearities, which can be exploited for the radii polynomial method in order to get error estimates in the $L^2$ sense. This allows the method to be applicable when the system or solution is not continuous, which is a limitation of other radii-polynomial-based methods. Numerical examples show how the method is implemented in practice. 

    \vspace{1em}
    

    \vspace{1em}
    
\end{abstract}

\section{Introduction}
\label{sect:introduction}

Rigorous computation is an area under active development since the 1980s \cite{1988nakao}. With an steady increase of computing power, numerical methods became viable tools for analyzing differential equations and gaining insight on structures such as invariant objects. However, standard numerical methods provide only approximations; the results are non-rigorous and cannot be used in formal proofs. They can only be used to gain insights on the true structures of the system. Moreover, some structures, such as bifurcations, may still be hidden even when using very accurate numerical methods.

Rigorous computational methods try to fill these gaps, providing mathematically valid estimates and bounds for truncation and rounding errors, and rigorously proving the existence of such hidden structures. Over the years, a number of such methods were developed, such as rigorous integration \cite{2002zgliczynski, capd, 2016miyaji}, Conley index methods \cite{2005day, 2008maier-paape}, self-consistent bounds \cite{2001zgliczynski, 2002zgliczynski2} and discretization methods \cite{2013nakao, 2020hashimoto}. A more thorough review can be seen in \cite{nakao-plum} and references therein. 

Of particular interest to us are the \emph{radii polynomials} approach \cite{2007day, 2007gameiro, 2010vanderberg, 2014lessard, 2016hungria, 2017gameiro, 2019breden, 2019reinhardt, 2020sheombarsing}. These methods recast the problem of investigating the existence of structures as finding solutions to functional equations, usually expanding the solutions in terms of a basis. Then, usual numerical methods are employed to find an approximate solution to these equations. Finally, using fixed point theorems, we can guarantee the existence of a true, rigorous solution of the functional equations within certain bounds of the numerical solution. The hypotheses of the fixed point theorems are in turn proven to be satisfied with the aid of the so-called radii polynomials.

In this work we present a new radii polynomial method employing Haar wavelets. While many other bases were already employed by this approach, such as Taylor series \cite{2019reinhardt}, Fourier series \cite{2010vanderberg} and Chebyshev polynomials \cite{2014lessard}, to the best of our knowledge, no attempt has been made to combine wavelet methods and rigorous computations. We also believe that this method can be can be a framework to build other methods upon, such as rigorous continuation methods and methods for partial differential equations.

Wavelets are functions that form an orthonormal basis for the $L^2$ function space. While wavelet theory was only relatively recently formalized, their special properties --- such as time and frequency localization --- made them widely applicable in many fields, such as signal processing and compression algorithms. This poses an interesting case, because most of the aforementioned radii polynomial methods were proposed to work with smooth functions, while our wavelet-based radii polynomial method works in more general settings. 

The radii polynomial method presented in this paper is based on the ideas of a numerical method for solving differential equations, the \emph{Haar wavelet method}, which was first proposed in \cite{chen-hsiao}. It is assumed that the highest-order derivative is expressed in terms of wavelets, and the solutions are given by the integral of the series, essentially rewriting the differential equation in its integral form. This allows the method to work with only Haar wavelets, the simplest wavelet available, and leads to a matrix representation of the integral operator. The simplicity of the Haar wavelet allows this matrix to be easily and recursively calculated. The original Haar wavelet method was further analyzed, developed and applied in several publications \cite{2009lepik, 2019oruc, 2013aziz, 2015siraj, 2020mehandiratta, 2015majak2}.

In our method, we find the functional equations for the radii polynomial method using the same expansions of the Haar wavelet method. This essentially transforms the differential equation into an integral one, and allow us to use a radii polynomial theorem similar to \cite{ginzburglandau}. Also, by using Haar wavelets, the integral operator and nonlinearities can be represented using infinite but recursive matrices, allowing us to make the estimates needed in the radii polynomial approach. 

The work is organized in the following way. In Section~\ref{sect:haar_wavelet} we introduce the Haar wavelet and its integral, and review some of their properties. In Section~\ref{sect:radii_poly} we introduce the radii polynomial method and prove the theorems that guarantee the existence of a true solution, provided that certain estimates are satisfied. In Section~\ref{sect:nonlinear} we study quadratic nonlinearities in order to prove estimates needed in the method. In Section~\ref{sect:examples} we present some examples illustrating the applications of the proposed method.
\section{The Haar wavelets and their integral}
\label{sect:haar_wavelet}

Here we introduce the Haar wavelet and its integral, which are one of the pillars of the proposed methods. 
Consider the space $L^2([0,1])$ of the square-integrable functions $f:[0,1] \to \R$ with respect to the usual Lebesgue measure. The Haar wavelets are a family of functions $\{\phi,\psi_{j,k}\} \subset L^2([0,1])$ defined, for $j = 0, 1, \ldots$ and $k = 0, 1, \ldots, 2^j-1$, by
\begin{equation}
\label{eq:haar_wavelet_definition}
    \phi(t) := 
    \begin{cases}
        1 &,\quad 0 \leq t < 1 \\
        0 &,\quad \text{otherwise}
    \end{cases}  \qquad \text{,} \qquad
    \psi_{j,k}(t) := 
    \begin{cases}
        {2^{j/2}}  &,\quad \frac{k}{2^j} \leq t < \frac{k+0.5}{2^j}\\[3pt]
        -{2^{j/2}} &,\quad \frac{k+0.5}{2^j} \leq t < \frac{k+1}{2^j}\\[3pt]
        0  &,\quad \text{otherwise}
    \end{cases}
\end{equation}

The Haar wavelets form an orthonormal basis for $L^2([0,1])$; the proof is in many standard texts in wavelet theory, see e.g. \cite{hernandez}. Hence any function $y(t) \in L^2([0,1])$ can be expanded into a unique Haar wavelet series
\begin{equation*}
    y(t) = c_1\phi(t) + \sum_{j=0}^\infty \sum_{k=0}^{2^j-1}  c_{j,k} \psi_{j,k}(t),
\end{equation*}
where $c_1 = \int_0^1 y(t)\phi(t)\,dt$, $c_{j,k} = \int_0^1 y(t)\psi_{j,k}(t)\,dt$, and the sum converges in $L^2([0,1])$. If we make $i=2^j+k+1$, then the sequence $(c_i)_{i=1}^\infty \in \ell^2(\R)$.


\begin{notation}
We can change between the ``one-index" and ``two-indices" notations, depending on which is more convenient in each case. One can be converted to the other by making, for all $i > 2$, 
\begin{equation*}
    i = 2^j + k + 1 \iff 
    \begin{aligned}
        j &= \lfloor \text{log}_2 i \rfloor \\
        k &= i - 2^j - 1
    \end{aligned}
\end{equation*}
where $\lfloor \cdot \rfloor$ is the floor function. The index $i=1$ is reserved for the scaling function $\phi$.
\end{notation}

Conversely, any $(c_i)_{i=1}^\infty \in \ell^2(\R)$ defines a unique $f(t) \in L^2([0,1])$ by making
\begin{equation*}
    f(t) = \sum_{i=1}^\infty c_i\psi_i(t) = c_1\phi(t) + \sum_{j=0}^\infty \sum_{k=0}^{2^j-1} c_{j,k}\psi_{j,k}(t).
\end{equation*}

This means that we can define an invertible operator $\mathcal{H}:L^2([0,1]) \to \ell^2(\R)$ defined element-wise as
\begin{equation*}
    (\mathcal{H}f)_i := \int_0^1 f(t)\psi_i(t)\,dt.
\end{equation*}
The inverse Haar transform $\mathcal{H}^{-1}$ is given by
\begin{equation*}
    (\mathcal{H}^{-1}\text{\textbf{c}})(t) := \text{\textbf{h}}^T(t)\text{\textbf{c}} = \sum_{i=1}^\infty c_i\psi_i(t) \quad, \quad \text{\textbf{h}}(t) := (\psi_1(t),\psi_2(t),...)
\end{equation*}
It is worth noting that the Haar transform is an isometry, due to the fact that the Haar wavelet system is an orthonormal basis of $L^2([0,1])$.

Another interesting and useful property of the Haar wavelets is what we call the ``nesting property'': 
\begin{proposition}[Nesting property]
    Let $\psi_{j,k}$ and $\psi_{m,n}$ be two Haar wavelets such that $j < m$. If $\text{supp }\psi_{j,k} \cap \text{supp }\psi_{m,n} \not= \emptyset$, then either $\text{supp }\psi_{m,n} \subseteq [\frac{k}{2^j}, \frac{k+0.5}{2^j})$ or $\text{supp }\psi_{m,n} \subseteq [\frac{k+0.5}{2^j}, \frac{k+1}{2^j})$.
\end{proposition}

The proof is simple and will be omitted; it consists in comparing the supports of the wavelets $\phi_{j,k}$, which are dyadic intervals of length $2^{-j}$, for different $j$. 
Intuitively, it means that, if the supports of two wavelets at different resolutions overlap, then the support of the ``finer" wavelet (that is, the higher-resolution one) is entirely nested within either the positive or the negative part of the ``coarser" wavelet.


For this work, we are also interested in the integral of the Haar wavelets, and how the integral relates to the wavelet themselves. The integrals of the Haar wavelet family in the interval $[0,1]$ are the triangular functions given by
\begin{equation}\label{eq:haar_wavelet_integral}
    w_1(t) = t \qquad \text{,} \qquad w_{j,k}(t) = 
    \begin{cases}
        {2^{j/2}}\left(t - \frac{k}{2^j}\right) &, \quad \frac{k}{2^j} \leq t \leq \frac{k+0.5}{2^j}\\[3pt]
        {2^{j/2}}\left(\frac{k+1}{2^j}-t \right) &,\quad \frac{k+0.5}{2^j} \leq t\leq \frac{k+1}{2^j}\\[3pt]
        0  &,\quad \text{otherwise}.
    \end{cases}
\end{equation}
We also extend the one-index notation to the Haar wavelet integrals $w_i$.


The Haar wavelet integrals are continuous functions in $[0,1]$; thus they are square-integrable in that interval, and can be expanded in Haar wavelet series themselves:
\begin{equation}\label{eq:haar_integral_series}
    w_i(t) = \sum_{l=1}^\infty P_{i,l} \psi_l(t) = P\text{\textbf{h}}^T(t) \ ,\  P_{i,l} := \int_0^1 \psi_l(t)w_i(t)\,dt.
\end{equation}

While $P$ is expressed by an infinite matrix, there is a recursive formula to compute it. We must first define the Haar matrix of our wavelet system:
\begin{definition}
    For a given resolution $J$, the \textbf{Haar matrix} $H_M$ of order $M = 2^{J+1}$ is given element-wise by
    \begin{equation}\label{eq:haar_matrix_element_definition}
        (H_M)_{p,q} := \psi_p(t_q)
    \end{equation}
    where $\displaystyle t_q = \frac{q-0.5}{M}$, for $q = 1,\dots,M$.
    
    The \textbf{discrete Haar transform matrix} $HT_M$ of order $M = 2^{J+1}$ is defined by
    \begin{equation}\label{eq:haar_transform_matrix_definition}
        HT_M := \frac{1}{\sqrt{M}}H_M
    \end{equation}
\end{definition}

A particularly important fact is that $HT_M$ is unitary for all $J$, and hence $H_M$ is invertible and $(H_M)^{-1} = \frac{1}{M} H_M^T$.

\begin{theorem}\label{thm:P_recursive_formula}
The infinite matrix $P$ can be recursively calculated as
\begin{equation}\label{eq:PJ_rec_formula}
    P_1 = \frac{1}{2} \qquad,\qquad P_{2m} = 
    \begin{bmatrix}
        P_m & -\displaystyle\frac{1}{4\sqrt{m^3}}H_m \\[0.1in]
        \displaystyle\frac{1}{4\sqrt{m^3}}H_m^T & 0_{m \times m}
    \end{bmatrix}
\end{equation}
for $m = 2^j$ and $j = 0,1,2,\dots$.
\end{theorem}

The proof for this formula is in \cite{chen-hsiao}, with some modifications to account for the fact that we are using normalized wavelets.

Let us prove that $P^T\text{\textbf{c}} \in \ell^2(\R)$ for all $\text{\textbf{c}} \in \ell^2(\R)$. We first need to define some projections. Given a resolution level $J$, let $M = 2^{J+1}$ as before, and define the projection $\Pi_M$ as
\begin{equation}
    \begin{aligned}
        \Pi_M: \ell^2(\R) &\to \R^M \\
        \text{\textbf{c}} &\mapsto (c_1,...,c_M).
    \end{aligned}
\end{equation}
We identify the vector $(c_1,...,c_M) \in \R^M$ with its infinite-dimensional counterpart $(c_1,...,c_M,0,0,...) \in \ell^2(\R)$. We also define the projection $\Pi_\infty \in B(\ell^2(\R))$ as 
\begin{equation*}
    \Pi_\infty\text{\textbf{c}} := (I - \Pi_M)\text{\textbf{c}}.
\end{equation*}

\begin{notation}
    Given $\text{\textbf{c}} \in \ell^2(\R)$, we sometimes divide it into blocks of length $2^n$, $n = 0,1,2,...$ as
    \begin{equation}\label{eq:vec_block_division}
        \text{\textbf{c}} = \left( \text{\textbf{c}}^*_0, \text{\textbf{c}}^*_1, \text{\textbf{c}}^*_2, \text{\textbf{c}}^*_4, ... \right)^T \ , \text{ where } \text{\textbf{c}}^*_{0} = c_1 \text{ , } \text{\textbf{c}}^*_{2^n} = \left(c_{2^{n-1}+1}, c_{2^{n-1}+2}, ..., c_{2^n} \right)^T.
    \end{equation}
    Also, given a matrix $A$, we denote $A_{m_1:m_2,n_1:n_2}$ the submatrix of $A$ given by
    \begin{equation*}
        A_{m_1:m_2,n_1:n_2} = 
        \begin{bmatrix}
            A_{m_1,n_1} & \cdots & A_{m_1,n_2} \\ 
            \vdots &  & \vdots \\
            A_{m_2,n_1} & \cdots & A_{m_2,n_2}
        \end{bmatrix}
    \end{equation*}
    If we wish to take all rows or all columns of $A$, we denote $A_{*,n_1:n_2}$ and $A_{m_1:m_2,*}$, respectively. Lastly, to reduce notation clutter, we denote $\text{\textbf{c}}_M = \Pi_M\text{\textbf{c}}$ and $\text{\textbf{c}}_\infty = \Pi_\infty\text{\textbf{c}}$.
\end{notation}

\begin{proposition}\label{prop:PTc_etimates}
    For $\text{\normalfont{\textbf{c}}} \in \ell^2(\R)$, $M = 2^{J+1}$, $J = 0,1,2,...$ and the projections as defined before,
    \begin{align*}
        \left\|\Pi_M P^T \text{\normalfont{\textbf{c}}}\right\|_{\ell^2} \leq \frac{1}{\sqrt{3}} \left(4 - \frac{1}{2^{2J+2}}\right)^\frac{1}{2} \|\text{\normalfont{\textbf{c}}}\|_{\ell^2} \quad\quad,\quad\quad
        \left\|\Pi_\infty P^T \text{\normalfont{\textbf{c}}}\right\|_{\ell^2} \leq \frac{1}{\sqrt{3}} \frac{\|\text{\normalfont{\textbf{c}}}\|_{\ell^2}^2}{2^{J+1}} 
    \end{align*}
    In particular, letting $J \to \infty$, $P^T \text{\normalfont{\textbf{c}}} \in \ell^2(\R)$ and $P^T: \ell^2(\R) \to \ell^2(\R)$ is a bounded linear operator. 
\end{proposition} 

\begin{proof}
From the structure of $P$ in Theorem~\ref{thm:P_recursive_formula} and the block representation of $\text{\textbf{c}}$, we have
\begin{equation}\label{eq:PTc_expression}
    \left(P^T\text{\textbf{c}}\right)_{2^j}^* = -\frac{1}{2^{\frac{3j}{2} + 2}} H_{2^j}^T (\Pi_{2^j}\text{\textbf{c}}) + \sum_{q=j+1}^\infty \frac{1}{2^{\frac{3q}{2}+2}} H_{2^q} \text{\textbf{c}}^*_{2^q}
\end{equation}
Recalling that $H_M = \sqrt{M} HT_M$ and $HT_M$ is a unitary matrix, we can bound the above term by
\begin{align}\label{eq:PTc_block_estimate_infinite_vector}
    \left\|\left(P^T\text{\textbf{c}}\right)_{2^j}^*\right\|_{\ell^2} &\leq \frac{1}{2^{j+2}} \left\|HT_{2^j}^T (\Pi_{2^j}\text{\textbf{c}})\right\|_{\ell^2} + \sum_{q=j+1}^\infty \frac{1}{2^{q+2}} \left\|HT_{2^q} \text{\textbf{c}}^*_{2^q}\right\|_{\ell^2} \notag \\
    &\leq \frac{\|\text{\textbf{c}}\|_{\ell^2}}{2^{j+2}} + \sum_{q=j+1}^\infty \frac{1}{2^{q+2}} \left\|\text{\textbf{c}}^*_{2^q}\right\|_{\ell^2} \leq \frac{\|\text{\textbf{c}}\|_{\ell^2}}{2^{j+1}}
\end{align}
The only term left is $\left(P^T\Bar{\text{\textbf{c}}}\right)_{1}$. We can bound it with
\begin{align}\label{eq:PTc_1st_element_estimate}
    \left| \left(P^T\Bar{\text{\textbf{c}}}\right)_1 \right| &= \left| \frac{1}{2}c_1 - \sum_{q=0}^\infty \frac{1}{2^{\frac{3q}{2}+2}} H_{2^q} \text{\textbf{c}}^*_{2^q} \right| \leq \frac{1}{2}\left|c_1\right| + \sum_{q=0}^\infty \frac{1}{2^{q+2}} \left\|HT_{2^q} \text{\textbf{c}}^*_{2^q}\right\|_{\ell^2} \leq \left\|\text{\textbf{c}}\right\|_{\ell^2}
\end{align}
Thus, the norm of $\Pi_M P^T \text{\normalfont{\textbf{c}}}$ can be estimated by
\begin{align*}
    \left\|\Pi_M P^T \text{\normalfont{\textbf{c}}}\right\|_{\ell^2}^2 &\leq \left|\left(P^T\text{\textbf{c}}\right)_1\right|^2 + \sum_{j=0}^J \left\|\left(P^T\text{\textbf{c}}\right)_{2^j}^*\right\|_{\ell^2}^2 = \left(4 - \frac{1}{2^{2J+2}}\right) \frac{\|\text{\textbf{c}}\|_{\ell^2}^2}{3}
\end{align*}
Analogously, the norm of $\Pi_\infty P^T \text{\normalfont{\textbf{c}}}$ is bounded by
\begin{align*}
    \left\|\Pi_\infty P^T \text{\normalfont{\textbf{c}}}\right\|_{\ell^2}^2 &\leq \sum_{j=J+1}^{\infty} \left\| \left(P^T\Bar{\text{\textbf{c}}}\right)_{2^j}^* \right\|_{\ell^2}^2 \leq \sum_{j=J+1}^\infty \frac{\|\text{\textbf{c}}\|_{\ell^2}^2}{2^{2j+2}} = \frac{1}{3} \frac{\|\text{\textbf{c}}\|_{\ell^2}^2}{2^{2J+2}}
\end{align*}
\end{proof}
\section{The radii polynomial approach}
\label{sect:radii_poly}

In this section, we introduce the radii polynomial \cite{2007day, 2007gameiro, 2010vanderberg, 2014lessard, 2016hungria, 2017gameiro} approach for rigorous computation. Consider an initial value problem 
\begin{equation}\label{eq:autonomous_differential_equation}
    \begin{cases}
        \Dot{u} = f(u,t) \\
        u(0) = u_0
    \end{cases}
\end{equation}
and suppose we find a numerical, approximate solution $\Bar{u}(t)$. Our aim is to prove the existence of a true solution $\Tilde{u}(t)$ in some neighborhood of $\Bar{u}(t)$. This is done using the radii polynomial method. 

For our work, suppose that $\dot{u}(t) \in L^2([0,1])$. Then, we can write $\dot{u}$ and $u$ using the Haar wavelet and its integral as
\begin{equation}\label{eq:solution_expansion}
    \dot{u}(t) = \sum_{i=1}^\infty c_i \psi_i(t) \quad,\quad u(t) = \sum_{i=1}^\infty c_i w_i(t) + u_0
\end{equation}
Substituting back into the differential equation \eqref{eq:autonomous_differential_equation} and taking the Haar transform,
\begin{equation}\label{eq:def_F}
    F(\text{\textbf{c}}) := \text{\textbf{c}} - \mathcal{H}\left(f\left( t,\sum_{i=1}^\infty c_i w_i(t) + u_0 \right)\right) = 0
\end{equation}
Thus, we have a map $F:\ell^2(\R) \to \ell^2(\R)$ such that finding a solution of the differential equation implies finding a zero of $F$. Conversely, due to the uniqueness of the wavelet series that represents $\dot{u}$ --- and consequently of the series that represents $u(t)$ --- finding a zero of $F$ is equivalent to finding a solution to \eqref{eq:autonomous_differential_equation}.

Now we recast the problem of finding the zeros of $F$ to finding a fixed point of a map $T: \ell^2(\R) \to \ell^2(\R)$ near the numerical solution $\Bar{u}$. This is done by showing that $T$ is a contraction near $\Bar{u}$. First, define the operator $A: \ell^2(\R) \to \ell^2(\R)$ as 
\begin{equation*}
    A\text{\textbf{x}} = A_M\Pi_M\text{\textbf{x}} + \Pi_\infty\text{\textbf{x}},
\end{equation*}
where $A_M: \R^M \to \R^M$ is a finite-dimensional, computational approximation for the inverse $D(\Pi_M F(\Bar{\text{\normalfont{\textbf{x}}}}))^{-1}$, with $\Bar{\text{\normalfont{\textbf{x}}}}$ the solution numerically obtained. Then we can define the fixed-point map $T: \ell^2(\R) \to \ell^2(\R)$ by
\begin{equation}\label{eq:T_fixed_point}
    T(\text{\textbf{x}}) := \text{\textbf{x}} - AF(\text{\textbf{x}}) = (\Pi_M - A_M \Pi_M F)(\text{\textbf{x}}) + \Pi_\infty(\text{\textbf{x}}-F(\text{\textbf{x}}))
\end{equation}
Its derivative, which is used for the radii polynomial method, is given by
\begin{equation}\label{eq:DT_fixed_point}
    DT(\text{\textbf{x}}) = \Pi_M - A_M \Pi_MDF(\text{\textbf{x}}) + \Pi_\infty(I-DF(\text{\textbf{x}}))
\end{equation}

\begin{remark}
    Since many of the matrices in this work are block matrices, one can compute their inverse as
    \begin{equation*}
        \begin{bmatrix}
            A & B \\
            C & D
        \end{bmatrix}^{-1}
        =
        \begin{bmatrix}
            A^{-1} + A^{-1}B(D-CA^{-1}B)^{-1}CA^{-1} & -A^{-1}B(D-CA^{-1}B)^{-1} \\
            -(D-CA^{-1}B)^{-1}CA^{-1} & (D-CA^{-1}B)^{-1}
        \end{bmatrix}
    \end{equation*}
    provided that the matrices $A$ and $D-CA^{-1}B$ are invertible. Calculating the inverse with this formula can be faster than directly inverting the full matrix $D(\Pi_M F(\Bar{\text{\normalfont{\textbf{x}}}}))$.
\end{remark}

\begin{notation}
    In order to help visualize the operators and reduce clutter in notation, we employ the following ``block matrix" notation for an operator $C:\ell^2(\R) \to \ell^2(\R)$ whenever it is convenient:
    \begin{equation}\label{eq:finite_infinite_decomposition}
        C\text{\textbf{x}} = 
        \begin{bmatrix} 
            C_M & C_{M,\infty} \\
            C_{\infty,M} & C_\infty
        \end{bmatrix}
        \begin{pmatrix}
            \Pi_M \text{\textbf{x}} \\
            \Pi_\infty \text{\textbf{x}}
        \end{pmatrix}.
    \end{equation}
    We refer to this as the \textbf{finite-infinite decomposition}.
\end{notation}

Intuitively, the first term of \eqref{eq:T_fixed_point} is a Newton-like map for the finite terms we computed numerically. In the second term, we hope the ``tail" of $F$ will contract to zero by itself --- which is what happens with quadratic nonlinearities. This is all motivated by the fact that, as $M$ increases, the new elements of the matrix $P_M$ become smaller. 

We now formally prove that a fixed point of $T$ corresponds to a zero of $F$:

\begin{proposition}\label{prop:fixed_point_zero_lin_eq}
    Suppose the map $T$ as defined above in \eqref{eq:T_fixed_point} is a contraction in some closed neighborhood of $\ell^2(\R)$.  Then $T$ has a unique fixed point $\Tilde{\text{\normalfont{\textbf{c}}}}$ in that neighborhood. Moreover, $\Tilde{\text{\normalfont{\textbf{c}}}}$ is a fixed point of $T$ if and only if it is a zero of $F$ as defined earlier in \eqref{eq:def_F}.
\end{proposition}

\begin{proof}
Since $T$ is a contraction in a closed neighborhood of $\ell^2(\R)$, the Banach Fixed Point Theorem guarantees that it has a fixed point $\Tilde{\text{\textbf{c}}}$ in the same neighborhood. Also, if $\Tilde{\text{\textbf{c}}}$ is a zero of $F$, then a straightforward calculation shows that it is a fixed point of $T$. 

It remains to prove that the fixed point $\Tilde{\text{\textbf{c}}}$ is a zero of $F$. By \eqref{eq:T_fixed_point}
\begin{equation*}
    T(\Tilde{\text{\textbf{c}}})-\Tilde{\text{\textbf{c}}} = 0 = A_M \Pi_M F(\Tilde{\text{\textbf{c}}}) + \Pi_\infty F(\Tilde{\text{\textbf{c}}}).
\end{equation*}
Since $A_M \Pi_M F(\Tilde{\text{\textbf{c}}}) \in \Pi_M(\ell^2(\R))$, we have
\begin{align*}
    \Pi_\infty F(\Tilde{\text{\textbf{c}}}) = 0 \quad,\quad A_M \Pi_M F(\Tilde{\text{\textbf{c}}}) = 0
\end{align*}
and since $A_M$ is invertible, then $\Pi_M F(\Tilde{\text{\textbf{c}}}) = 0$ as well. Thus $F(\Tilde{\text{\textbf{c}}}) = 0$.
\end{proof}

To prove that $T$ is actually a contraction near of our numerical solution $\Bar{\text{\textbf{c}}}$, we use the radii polynomials. First, we define the closed neighborhood  
\begin{equation}\label{eq:closed_balls}
    \overline{B_\omega(\Bar{\text{\textbf{c}}},r)} = \left\{ \text{\textbf{y}} \in \ell^2(\R) : \|\Pi_M(\text{\textbf{y}}-\Bar{\text{\textbf{c}}})\| \leq \omega r \text{ and } \|\Pi_\infty(\text{\textbf{y}}-\Bar{\text{\textbf{c}}})\| \leq (1-\omega) r \right\}
\end{equation}
in which $T$ will be a contraction. $\omega \in (0,1)$ is a ``trade-off parameter'': we can loosen the radius in the infinite part, at the cost of tightening the radius in the finite part, and vice-versa. For the next calculations, we assume $\omega$ is fixed, though in practice it is chosen later.

Next, we need bounds $Y_M$ and $Y_\infty$, and polynomials $Z_M(r)$ and $Z_\infty(r)$ such that
\begin{gather}
    \|\Pi_M(T(\Bar{\text{\textbf{c}}})-\Bar{\text{\textbf{c}}})\|_{\ell^2} \leq Y_M \label{eq:YM_bound_def} \\
    \|\Pi_\infty(T(\Bar{\text{\textbf{c}}})-\Bar{\text{\textbf{c}}})\|_{\ell^2} \leq Y_\infty \label{eq:Yinfty_bound_def} \\
    \sup_{x_1,x_2 \in \overline{B(r)}} \|\Pi_M(DT(\Bar{\text{\textbf{c}}} + x_1)x_2)\| \leq Z_M(r)r \label{eq:ZM_bound_def} \\
    \sup_{x_1,x_2 \in \overline{B(r)}} \|\Pi_\infty(DT(\Bar{\text{\textbf{c}}} + x_1)x_2)\| \leq Z_\infty(r)r \label{eq:Zinfty_bound_def}
\end{gather}
Then, we can define the radii polynomials as
\begin{equation}\label{eq:radii_poly_def}
    \begin{aligned}
        p_M(r) &:= Z_M(r)r - \omega r + Y_M  \quad,\quad
        p_\infty(r) &:= Z_\infty(r)r - (1-\omega)r + Y_\infty
    \end{aligned}
\end{equation}

\begin{theorem}\label{thm:radii_poly_fixed_point}
    Consider the radii polynomials as defined in \eqref{eq:radii_poly_def}. If there exists an $r_0 > 0$ such that $p_M(r_0) < 0$ and $p_\infty(r_0) < 0$, then there exists a unique $\Tilde{\text{\normalfont{\textbf{c}}}} \in B_\omega(\Bar{\text{\normalfont{\textbf{c}}}},r_0)$ such that $T(\Tilde{\text{\normalfont{\textbf{c}}}}) = \Tilde{\text{\normalfont{\textbf{c}}}}$.
\end{theorem}

\begin{proof}
Due to the Banach Fixed Point Theorem, we only need to prove that $T\left(\overline{B_\omega(\Bar{\text{\textbf{c}}},r_0)}\right) \subseteq \overline{B_\omega(\Bar{\text{\textbf{c}}},r_0)}$ and that $T$ is a contraction when restricted to $\overline{B_\omega(\Bar{\text{\textbf{c}}},r_0)}$.

We first prove that $T$ restricted to $\overline{B_\omega(\Bar{\text{\textbf{c}}},r_0)}$ is a contraction. If $y_1,y_2 \in \overline{B_\omega(\Bar{\text{\textbf{c}}},r_0)}$, then by the Mean Value Theorem
\begin{align*}
    \|T(y_1) - T(y_2)\|_{\ell^2} &\leq \sup_{x\in \overline{B_\omega(\Bar{\text{\textbf{c}}},r_0)}}\|DT(x)\| \|y_1-y_2\|_{\ell^2} = \sup_{x_1\in \overline{B(r_0)}} \|DT(\Bar{\text{\textbf{c}}} + x_1)\| \|y_1-y_2\|_{\ell^2}
\end{align*}
Hence, we must show that $\|DT(\Bar{\text{\textbf{c}}} + x_1)\| < 1$ for $x_1 \in \overline{B(r_0)}$. Observe that
\begin{align*}
    \sup_{x_1\in \overline{B(r_0)}} \|DT(\Bar{\text{\textbf{c}}} + x_1)\| &= \frac{1}{r_0}\sup_{x_1,x_2 \in \overline{B(r_0)}} \|DT(\Bar{\text{\textbf{c}}} + x_1)x_2\|_{\ell^2} \\
    &\hspace{-2em} \leq \frac{1}{r_0} \bigg(\sup_{x_1,x_2 \in \overline{B(r_0)}} \|\Pi_M DT(\Bar{\text{\textbf{c}}} + x_1)x_2\|_{\ell^2} + \sup_{x_1,x_2 \in \overline{B(r_0)}} \|\Pi_\infty DT(\Bar{\text{\textbf{c}}} + x_1)x_2\|_{\ell^2} \bigg) \\
    &\hspace{-2em} \leq Z_M(r_0) + Z_\infty(r_0)
\end{align*}
So we have
\begin{equation*}
    \|T(y_1) - T(y_2)\|_{\ell^2} \leq (Z_M(r_0) + Z_\infty(r_0)) \|y_1-y_2\|_{\ell^2}
\end{equation*}
But since $p_M(r_0) < 0$ and $p_\infty(r_0) < 0$,
\begin{align*}
    Z_M(r_0)r_0 + Z_\infty(r_0)r_0 - r_0 &\leq (Z_M(r_0)r_0-\omega r_0+Y_M) +  (Z_\infty(r_0)r_0-(1-\omega)r_0+Y_\infty) \\
    &\hspace{2em} = p_M(r_0) + p_\infty(r_0) < 0
\end{align*}
Thus $Z_M(r_0) + Z_\infty(r_0) < 1$, and $T$ restricted to $\overline{B_\omega(\Bar{\text{\textbf{c}}},r_0)}$ is a contraction.

Now we must prove that $T\left(\overline{B_\omega(\Bar{\text{\textbf{c}}},r_0)}\right) \subseteq \overline{B_\omega(\Bar{\text{\textbf{c}}},r_0)}$. If $y \in \overline{B_\omega(\Bar{\text{\textbf{c}}},r_0)}$, then
\begin{align*}
    \|\Pi_M(T(y)-\Bar{\text{\textbf{c}}})\|_{\ell^2} &\leq \|\Pi_M(T(y)-T(\Bar{\text{\textbf{c}}}))\|_{\ell^2} + \|\Pi_M(T(\Bar{\text{\textbf{c}}}) -  \Bar{\text{\textbf{c}}})\|_{\ell^2} \\
    &\leq Z_M(r_0)\|y-\Bar{\text{\textbf{c}}}\|_{\ell^2} + Y_M \leq Z_M(r_0)r_0 + Y_M < \omega r_0
\end{align*}
and similarly for $\Pi_\infty(T(y)-\Bar{\text{\textbf{c}}})$
\begin{align*}
    \|\Pi_\infty(T(y)-\Bar{\text{\textbf{c}}})\|_{\ell^2} &\leq \|\Pi_\infty(T(y)-T(\Bar{\text{\textbf{c}}}))\|_{\ell^2} + \|\Pi_\infty(T(\Bar{\text{\textbf{c}}}) -  \Bar{\text{\textbf{c}}})\|_{\ell^2} \\
    &\leq Z_\infty(r_0)\|y-\Bar{\text{\textbf{c}}}\|_{\ell^2} + Y_\infty \leq Z_\infty(r_0)r_0 + Y_\infty < (1-\omega) r_0
\end{align*}
and hence $T(y) \in \overline{B_\omega(\Bar{\text{\textbf{c}}},r_0)}$.
\end{proof}

Thus, if the radii polynomial method is successful in finding an $r_0$, then the solution $\Bar{\text{\textbf{c}}}$ found by the numerical method is ``close" to the wavelet coefficients of true solution $\Tilde{\text{\textbf{c}}}$ in the $\ell^2(\R)$ sense, that is, $\|\Bar{\text{\textbf{c}}} - \Tilde{\text{\textbf{c}}}\|_{\ell^2} \leq r_0$. Or equivalently, the numerical approximation $\Bar{u}(t) = \Bar{\text{\textbf{c}}}^T\text{\textbf{h}}(t)$ is ``close" to the true solution $\Tilde{u}(t)$ in the $L^2([0,1])$ sense, that is, $\|\Bar{u}-\Tilde{u}\|_{L^2} \leq r_0$.
\section{Nonlinear terms}
\label{sect:nonlinear}

In this section we study quadratic nonlinearities in more depth. This may seem restrictive, but there are many interesting systems involving those, such as the Lorenz system. Furthermore, we believe that estimates for higher nonlinearities can be computed with similar techniques.

Consider two functions $u$ and $v$ such that $\dot{u},\dot{v}\in L^2([0,1])$. Their expansions into Haar wavelet integrals as in \eqref{eq:solution_expansion} are
\begin{align*}
    u(t) = u_0 + \sum_{i=1}^\infty c_iw_i(t) = u_0 + \text{\textbf{w}}^T(t)\text{\textbf{c}} \ , \ v(t) = v_0 + \sum_{i=1}^\infty d_iw_i(t) =  v_0 + \text{\textbf{w}}^T(t)\text{\textbf{d}}
\end{align*}
Thus, considering the product $u(t)v(t)$, we have
\begin{align*}
    u(t)v(t) = u_0v_0 + u_0\text{\textbf{w}}^T(t)\text{\textbf{d}} + v_0\text{\textbf{w}}^T(t)\text{\textbf{c}} + (W(t))(\text{\textbf{c}},\text{\textbf{d}})
\end{align*}
where, for $t \in [0,1]$, $(W(t))(\text{\textbf{c}},\text{\textbf{d}}) := \text{\textbf{c}}^T\text{\textbf{w}}(t)\text{\textbf{w}}^T(t)\text{\textbf{d}}$. A crucial observation is that, for any given $t \in [0,1]$, $W(t)$ \emph{is a symmetric bilinear form}. 

The next theorem shows that $W(t)(\text{\textbf{c}},\text{\textbf{d}}) \in L^2([0,1])$ for any pair $\text{\textbf{c}},\text{\textbf{d}} \in \ell^2(\R)$, which allows us to calculate its Haar transform and use the radii polynomial methods developed in Section~\ref{sect:radii_poly}. However, its proof is lengthy and will be left to \ref{apdx:biform} for clarity.

\begin{theorem}\label{thm:biform_haar_transform}
    The bilinear form $\left(W(\text{\normalfont{\textbf{c}}},\text{\normalfont{\textbf{d}}})\right)(t) := \text{\normalfont{\textbf{c}}}^T \text{\normalfont{\textbf{w}}}(t) \text{\normalfont{\textbf{w}(t)}} \text{\normalfont{\textbf{d}}}$ is bounded in $L^2([0,1])$ for all $\text{\normalfont{\textbf{c}}},\text{\normalfont{\textbf{d}}} \in \ell^2(\R)$, that is, there exists $C > 0$ such that
    \begin{equation*}
        \left\|W(\text{\normalfont{\textbf{c}}},\text{\normalfont{\textbf{d}}})\right\|_{L^2} \leq C \left\|\text{\normalfont{\textbf{c}}}\right\|_{\ell^2} \left\|\text{\normalfont{\textbf{d}}}\right\|_{\ell^2} \text{ for all } \text{\normalfont{\textbf{c}}},\text{\normalfont{\textbf{d}}} \in \ell^2(\R).
    \end{equation*}
    Additionally, making $\text{\normalfont{\textbf{a}}} := P^T\text{\normalfont{\textbf{c}}}$ and $\text{\normalfont{\textbf{b}}} := P^T\text{\normalfont{\textbf{d}}}$, we can write 
    \begin{equation}
        \text{\normalfont{\textbf{c}}}^T \text{\normalfont{\textbf{w}}}(t) \text{\normalfont{\textbf{w}(t)}} \text{\normalfont{\textbf{d}}} = \text{\normalfont{\textbf{a}}}^T\Omega(t)\text{\normalfont{\textbf{b}}} + \text{\normalfont{\textbf{a}}}^T\Omega^T(t)\text{\normalfont{\textbf{b}}} + \text{\normalfont{\textbf{a}}}^T\Theta(t)\text{\normalfont{\textbf{b}}}
    \end{equation}
    where the operator $\Omega(t)$ can be recursively defined as
    \begin{equation}\label{eq:biform_Omega}
        \Omega_1(t) = 0 \quad,\quad \Omega_{2m}(t) = 
    \begin{bmatrix}
        \Omega_{m}(t) & \Upsilon_m(t) \\[0.1in]
        0_{m \times m} & 0_{m \times m}
    \end{bmatrix}
    \end{equation}
    with $\Upsilon_m(t)$ being a $m \times m$ matrix defined element-wise as
    \begin{equation*}
        (\Upsilon_m)_{i,l}(t) = (H_m)_{i,l}\psi_{m+l}(t),
    \end{equation*} 
    and the operator $\Theta(t)$ can be represented as an infinite diagonal matrix given by
    \begin{equation}\label{eq:biform_Theta}
        \def\arraystretch{1.2}
        \Theta(t) = \left[
        \begin{array}{cccc}
            \psi_1^2(t) & \phantom{0} & \phantom{0} & \phantom{0} \\
            \phantom{0} & \psi_2^2(t) & \phantom{0} & \phantom{0} \\
            \phantom{0} & \phantom{0} & \psi_3^2(t) & \phantom{0} \\
            \phantom{0} & \phantom{0} & \phantom{0} & \ddots \\
        \end{array}
        \right]
    \end{equation}
    with zeros omitted for clarity. With the terms defined as above, 
    \begin{equation*}
        \mathcal{H}(\text{\normalfont{\textbf{c}}}^T \text{\normalfont{\textbf{w}}}(t) \text{\normalfont{\textbf{w}(t)}} \text{\normalfont{\textbf{d}}}) = \mathcal{H} (\text{\normalfont{\textbf{a}}}^T\Omega(t)\text{\normalfont{\textbf{b}}}) + \mathcal{H} (\text{\normalfont{\textbf{a}}}^T\Omega(t)\text{\normalfont{\textbf{b}}}) + \mathcal{H} (\text{\normalfont{\textbf{a}}}^T\Omega(t)\text{\normalfont{\textbf{b}}})
    \end{equation*}
    and
    \begin{equation}\label{haar_transforms_biform}
        \begin{aligned}             
            \mathcal{H} (\text{\normalfont{\textbf{a}}}^T\Omega(t)\text{\normalfont{\textbf{b}}}) &= (\Tilde{\Omega}^T \text{\normalfont{\textbf{a}}}) \odot \text{\normalfont{\textbf{b}}} \\
            \mathcal{H}(\text{\normalfont{\textbf{a}}}^T\Omega^T(t)\text{\normalfont{\textbf{b}}}) &= (\Tilde{\Omega}^T \text{\normalfont{\textbf{b}}}) \odot \text{\normalfont{\textbf{a}}} \\
            \mathcal{H}(\text{\normalfont{\textbf{a}}}^T\Theta(t)\text{\normalfont{\textbf{b}}}) &= \Gamma^T (\text{\normalfont{\textbf{a}}}\odot\text{\normalfont{\textbf{b}}})
        \end{aligned}
    \end{equation}
    where $\Tilde{\Omega}^T$ and $\Tilde{\Gamma}^T$ are recursively defined as
    \begin{equation*}
    	\begin{aligned}
    		\Tilde{\Omega}_1 = 0 \quad,\quad \Tilde{\Omega}_{2m} = 
    		\begin{bmatrix}
    			\Tilde{\Omega}_{m} & H_m \\
    			0_m & 0_m
    		\end{bmatrix}
    		\\
    		\Gamma_1 = 1 \quad,\quad \Gamma_{2m} = 
    		\begin{bmatrix}
    			\Gamma_m & 0_m \\
    			H_m^T & 0_m
    		\end{bmatrix}
	   	\end{aligned}
        \quad, \text{ for } m = 2^j \text{ and } j=0,1,2,3,...
    \end{equation*}
\end{theorem}

We now present some estimates required for the radii polynomial method. The full proof for those estimates are lengthy and left \ref{apdx:estimates}. The main strategy consists in employing both the recursive block structures of the matrices from Theorems~\ref{thm:P_recursive_formula} and \ref{thm:biform_haar_transform} and the finite-infinite decomposition from \eqref{eq:finite_infinite_decomposition}. These estimates provide tighter bounds which increase the likelihood of finding an $r_0$ which satisfies Theorem~\ref{thm:radii_poly_fixed_point}. We believe that similar estimates may be applied for higher-degree polynomial nonlinearities.

\begin{proposition}\label{prop:estimates_finite_proj_finite_vectors}
    Given $\Bar{\text{\normalfont{\textbf{c}}}},\Bar{\text{\normalfont{\textbf{d}}}} \in \R^M$, $M = 2^{J+1}$ for some $J \geq 0$, and $\Bar{\text{\normalfont{\textbf{a}}}} = P^T\Bar{\text{\normalfont{\textbf{c}}}}$,  $\Bar{\text{\normalfont{\textbf{b}}}} = P^T\Bar{\text{\normalfont{\textbf{d}}}}$, the following estimates are valid:
    \begin{enumerate}[wide,label = \roman*)]
        \item $\Pi_M P^T\Bar{\text{\normalfont{\textbf{c}}}} = P_M^T \Bar{\text{\normalfont{\textbf{c}}}}$
        
        \item $\Pi_M \mathcal{H} (\Bar{\text{\normalfont{\textbf{a}}}}^T\Omega(t)\Bar{\text{\normalfont{\textbf{b}}}}) = (\Tilde{\Omega}_M^TP_M^T\Bar{\text{\normalfont{\textbf{c}}}}) \odot (P_M^T\Bar{\text{\normalfont{\textbf{d}}}})$
        
        \item $\Pi_M \mathcal{H} (\Bar{\text{\normalfont{\textbf{a}}}}^T\Theta(t)\Bar{\text{\normalfont{\textbf{b}}}}) = \Gamma^T_M (\Bar{\text{\normalfont{\textbf{a}}}}_M \odot \Bar{\text{\normalfont{\textbf{b}}}}_M) + \Gamma^T_{\infty,M} (\Bar{\text{\normalfont{\textbf{a}}}}_\infty \odot \Bar{\text{\normalfont{\textbf{b}}}}_\infty)$; moreover,
        \begin{equation*}
            \left\| \Gamma^T_\infty (\Bar{\text{\normalfont{\textbf{a}}}}_\infty \odot \Bar{\text{\normalfont{\textbf{b}}}}_\infty) \right\|_{\ell^2} \leq \frac{\sqrt{2} \, \|\Bar{\text{\normalfont{\textbf{c}}}}\|_{\ell^2} \|\Bar{\text{\normalfont{\textbf{d}}}}\|_{\ell^2} }{\left(4-\sqrt{2}\right)2^{\frac{3J}{2}+3}}.
        \end{equation*}
    \end{enumerate}
\end{proposition}

\begin{proposition}\label{prop:estimates_infinite_proj_finite_vectors}
    Given $\Bar{\text{\normalfont{\textbf{c}}}},\Bar{\text{\normalfont{\textbf{d}}}} \in \R^M$, $M = 2^{J+1}$ for some $J \geq 0$, and $\Bar{\text{\normalfont{\textbf{a}}}} = P^T\Bar{\text{\normalfont{\textbf{c}}}}$,  $\Bar{\text{\normalfont{\textbf{b}}}} = P^T\Bar{\text{\normalfont{\textbf{d}}}}$, the following estimates are valid:
    \begin{enumerate}[wide, label = \roman*)]
        
        \item $\displaystyle \left\|\Pi_\infty P^T\Bar{\text{\normalfont{\textbf{c}}}}\right\|_{\ell^2} \leq \frac{1}{\sqrt{3}} \frac{\left\|\Bar{\text{\normalfont{\textbf{c}}}}\right\|_{\ell^2}}{2^{J+2}}$
        
        \item $\displaystyle \left\| \Pi_\infty \mathcal{H} (\Bar{\text{\normalfont{\textbf{a}}}}^T\Omega(t)\Bar{\text{\normalfont{\textbf{b}}}}) \right\|_{\ell^2} \leq \frac{1}{\sqrt{3}} \frac{\left\|\Bar{\text{\normalfont{\textbf{c}}}}\right\|_{\ell^2} \left\|\Bar{\text{\normalfont{\textbf{d}}}}\right\|_{\ell^2}}{2^{2J+4}} $
        
        \item $\displaystyle \left\| \Pi_\infty \mathcal{H} (\Bar{\text{\normalfont{\textbf{a}}}}^T\Theta(t)\Bar{\text{\normalfont{\textbf{b}}}}) \right\|_{\ell^2} \leq \frac{1}{21\sqrt{7}} \frac{\left\|\Bar{\text{\normalfont{\textbf{c}}}}\right\|_{\ell^2} \left\|\Bar{\text{\normalfont{\textbf{d}}}}\right\|_{\ell^2} }{2^{3J+6}}.$
    \end{enumerate}
\end{proposition}

\begin{proposition}\label{prop:estimates_finite_proj_infinite_vectors}
    Given $\text{\normalfont{\textbf{x}}},\text{\normalfont{\textbf{y}}} \in \ell^2(\R)$ and $\Bar{\text{\normalfont{\textbf{c}}}} \in \R^M$,
    \begin{enumerate}[label=\roman*)]
        \item $\displaystyle \left\| (P_{\infty,M}^T\,\text{\normalfont{\textbf{y}}}_\infty) \right\|_{\ell^2} \leq \frac{\left\| \text{\normalfont{\textbf{y}}} \right\|_{\ell^2}}{2^{J+2}}$
        
        \item $\displaystyle \left\| \left(\Tilde{\Omega}^TP^T\right)_{M,\infty} \text{\normalfont{\textbf{y}}}_\infty \right\|_{\ell^2} \leq \left(1+\sqrt{2}\right) \frac{\left\|\text{\normalfont{\textbf{y}}}\right\|_{\ell^2}}{2^\frac{J+3}{2}}$
        
        \item $\displaystyle \left\| \Gamma_{\infty,M}^T(\Bar{\text{\normalfont{\textbf{a}}}}_\infty \odot \Pi_\infty P^T\text{\normalfont{\textbf{y}}}) \right\|_{\ell^2} \leq \frac{\sqrt{2}}{4-\sqrt{2}} \frac{\left\|\Bar{\text{\normalfont{\textbf{c}}}}\right\|_{\ell^2} \left\|\text{\normalfont{\textbf{y}}}\right\|_{\ell^2}}{2^{\frac{3J}{2}+3}}$
         
        \item $\displaystyle
            \left\| A_M\Pi_M \mathcal{H}(\text{\normalfont{\textbf{x}}}^TP\, \Omega(t) P^T \text{\normalfont{\textbf{y}}}) \right\|_{\ell^2} \leq K_1 \|\text{\normalfont{\textbf{x}}}\|_{\ell^2} \|\text{\normalfont{\textbf{y}}}\|_{\ell^2}$
         
        \item $\displaystyle \left\|A_M\Pi_M\mathcal{H}(\text{\normalfont{\textbf{x}}}^T P\Theta(t)P^T \text{\normalfont{\textbf{y}}})\right\|_{\ell^2} \leq K_2 \|\text{\normalfont{\textbf{x}}}\|_{\ell^2} \|\text{\normalfont{\textbf{y}}}\|_{\ell^2}$
    \end{enumerate}

    where
    \begin{align*}
        K_1 &= 
        \left\|A_M\,\text{\normalfont{diag}}\left(\|(P_M^T)_{i,*}\|_{\ell^2}\right)\, \Tilde{\Omega}_M^TP_M^T \right\| + \frac{\|A_M\Tilde{\Omega}_M^T P_M^T\|}{2^{J+2}} + \frac{(1+\sqrt{2})\|A_MP_M^T\|}{2^\frac{J+3}{2}} +  \frac{(1+\sqrt{2})\|A_M\|}{2^\frac{3J+7}{2}}\\
        K_2 &= \left\| A_M\Gamma_M^T \text{\normalfont{diag}}(\|(P_M^T)_{i,*}\|_\ell^2) P_M^T \right\| + \frac{\left\| A_M \Gamma_M^T P_M^T\right\|}{2^{J+1}} + \frac{\left\| A_M\Gamma_M^T \right\|}{2^{2J+4}} + \frac{\sqrt{2}\|A_M\|}{(4-\sqrt{2})2^{\frac{3J}{2}+4}}
    \end{align*}
\end{proposition}

\begin{proposition}\label{prop:estimates_infinite_proj_infinite_vectors}
    Given $\text{\normalfont{\textbf{x}}},\text{\normalfont{\textbf{y}}} \in \ell^2(\R)$ and $\Bar{\text{\normalfont{\textbf{c}}}} \in \R^M$, the following estimates are valid:
    \begin{enumerate}[label=\roman*)]
        
        \item $\displaystyle  \left\|\Pi_\infty\mathcal{H} (\Bar{\text{\normalfont{\textbf{c}}}}^T \text{\normalfont{\textbf{w}}}(t) \text{\normalfont{\textbf{w}}}^T(t) \text{\normalfont{\textbf{y}}}) \right\|_{\ell^2} \leq D_1 \frac{\left\|\Bar{\text{\normalfont{\textbf{c}}}}\right\|_{\ell^2} \left\|\text{\normalfont{\textbf{y}}}\right\|_{\ell^2}}{2^{\frac{3J}{2}}} \quad,\quad D_1 = \frac{1}{8\sqrt{7}} \left( 3+\frac{\sqrt{2}}{4} +  \frac{4}{4-\sqrt{2}} \right);$ 
        
        \item $\displaystyle \left\|\Pi_\infty\mathcal{H} (\text{\normalfont{\textbf{x}}}^T \text{\normalfont{\textbf{w}}}(t) \text{\normalfont{\textbf{w}}}^T(t) \text{\normalfont{\textbf{y}}}) \right\|_{\ell^2} \leq D_2 \frac{\left\|\text{\normalfont{\textbf{x}}}\right\|_{\ell^2} \left\|\text{\normalfont{\textbf{y}}}\right\|_{\ell^2}}{2^{\frac{3J}{2}+2}} \quad,\quad D_2 = \frac{8+6\sqrt{2}}{\sqrt{7}\left(4-\sqrt{2}\right)}$.
    \end{enumerate}
\end{proposition}

\section{Examples}
\label{sect:examples}

In this section we illustrate the implementation of our method by means of three examples: the logistic equation, the logistic equation with a discontinuous forcing term, and the Lorenz system. The method was implemented in MATLAB R2021b using the INTLAB package for interval arithmetic \cite{intlabdoc}. The files for these examples are available at \url{https://github.com/gknakassima/RigComp-HaarWavelet}.

\subsection{Logistic equation}\label{subsect:logistic}
As a first simple example we look at the logistic equation
\begin{equation}\label{eq:logistic_equation}
    \dot{u} = \lambda u(1 - u) \quad,\quad u(0) = u_0
\end{equation}
since it has a polynomial nonlinearity and its analytical solution is given by
$u(t) = \frac{u_0e^{\lambda t}}{1-u_0+u_0e^{\lambda t}}$. Using the expansions in \eqref{eq:solution_expansion}, we have the functional equation equivalent to \eqref{eq:logistic_equation}:
\begin{equation}\label{eq:logistic_inf_matrix_form}
    \text{\textbf{c}}^T\text{\textbf{h}}(t) - \lambda\left(\text{\textbf{c}}^T\text{\textbf{w}}(t) + u_0 - u_0^2 - 2u_0 \text{\textbf{c}}^T\text{\textbf{w}}(t) - (W(t))(\text{\textbf{c}},\text{\textbf{c}})\right) = 0.
\end{equation}

\subsubsection{Obtaining a numerical approximation}

In order to obtain a finite-dimensional approximation $\Bar{\text{\textbf{c}}}$ of the solution, we first consider a truncated version of our matrix equation. Given a resolution level $J > 0$ and making $M=2^{J+1}$, we apply the projection $\Pi_M$ to all sequences of \eqref{eq:logistic_inf_matrix_form}, obtaining
\begin{equation*}
    \text{\textbf{c}}_M^T\text{\textbf{h}}_M(t) - \lambda\left(\text{\textbf{c}}_M^T\text{\textbf{w}}_M(t) + u_0 - u_0^2 - 2u_0 \text{\textbf{c}}_M^T\text{\textbf{w}}_M(t) - (W_M(t))(\text{\textbf{c}}_M,\text{\textbf{c}}_M)\right) = 0
\end{equation*}
where, for $\text{\textbf{a}}_M,\text{\textbf{b}}_M \in \R^M$, $(W_M(t))(\text{\textbf{a}}_M,\text{\textbf{b}}_M) := \text{\textbf{a}}_M^T\text{\textbf{w}}_M(t)\text{\textbf{w}}_M^T(t)\text{\textbf{b}}_M$. Since the equation holds for all $t \in [0,1]$, we sample it at the times $t_l = \frac{l-0.5}{2^J}$ for $l = 1,...,M$. Recalling that 
\begin{equation*}
    H_M = [\text{\textbf{h}}_M(t_1), \text{\textbf{h}}_M(t_2),...,\text{\textbf{h}}_M(t_l),...,\text{\textbf{h}}_M(t_{2^{J+1}})],
\end{equation*}
we can organize the time samples in matrix form as
\begin{equation*}
    \text{\textbf{c}}_M^TH_M + \lambda(2u_0-1)\text{\textbf{c}}_M^TP_MH_M + \lambda(u_0^2-u_0)\text{\textbf{e}}^T + \lambda(\text{\textbf{W}}_M(\text{\textbf{c}}_M,\text{\textbf{c}}_M))^T = 0
\end{equation*}
with $\text{\textbf{e}}$ and $\text{\textbf{W}}_M(\text{\textbf{c}}_M,\text{\textbf{c}}_M)$ being $M \times 1$ vectors given by
\begin{equation*}
    \text{\textbf{e}} := (1,1,...,1)^T \quad,\quad  \big( \text{\textbf{W}}_M(\text{\textbf{c}}_M,\text{\textbf{c}}_M) \big)_i := (W_M(t_i))(\text{\textbf{c}}_M,\text{\textbf{c}}_M).
\end{equation*}
By transposing this system and multiplying by $(H_M^T)^{-1} = \frac{1}{M}H_M$ we finally arrive at the equation to solve numerically:
\begin{equation}\label{eq:logistic_numerical_equation}
    \text{\textbf{c}}_M + \lambda(2u_0-1)P_M^T\text{\textbf{c}}_M + \frac{\lambda(u_0^2-u_0)}{M}H_M\text{\textbf{e}} + \frac{\lambda}{M}H_M\text{\textbf{W}}_M(\text{\textbf{c}}_M,\text{\textbf{c}}_M) = 0.
\end{equation}

As this is a nonlinear equation, we use Newton's method. Define  $F_M: \R^M \to \R^M$ as
\begin{equation}\label{eq:logistic_F_newton_method}
    F_M(\text{\textbf{c}}_M) := \text{\textbf{c}}_M + \lambda(2u_0-1)P_M^T\text{\textbf{c}}_M + \frac{\lambda(u_0^2-u_0)}{M}H_M\text{\textbf{e}} + \frac{\lambda}{M}H_M\text{\textbf{W}}_M(\text{\textbf{c}}_M,\text{\textbf{c}}_M)
\end{equation}
Thus, we apply Newton's method by iteratively calculating
\begin{equation}\label{eq:newton_method}
    \text{\textbf{c}}^{p+1}_M = \text{\textbf{c}}^{p}_M - A_M(\text{\textbf{c}}^{p}_M) F_M(\text{\textbf{c}}^{p}_M)
\end{equation}
where $A_M(\text{\textbf{c}}^{p}_M)$ is a numerical approximation for $\Big(DF_M(\text{\textbf{c}}^{p}_M)\Big)^{-1}$ and $\text{\textbf{c}}^{p}_M$ is the result from the $p$-th iteration.

\subsubsection{Estimates for the radii polynomials}

Here, we provide the bounds for Theorem~\ref{thm:radii_poly_fixed_point}. The maps $T$ and $DT$ for the fixed point theorem are given by \eqref{eq:T_fixed_point} and \eqref{eq:DT_fixed_point}, respectively. For this example, the map $F:\ell^2(\R) \to \ell^2(\R)$ is given as
\begin{equation}\label{eq:logistic_F_expression}
    F(\text{\textbf{x}}) := \lambda(u_0^2 - u_0)e_1 + \text{\textbf{x}} + \lambda(2u_0 - 1) P^T\text{\textbf{x}} + \lambda \mathcal{H}\left(\text{\textbf{x}}^T\text{\textbf{w}}(t)\text{\textbf{w}}^T(t)\text{\textbf{x}}\right).
\end{equation}
and its derivative $DF(\text{\textbf{x}})$ applied to $\text{\textbf{y}}\in \ell^2(\R)$ is given by
\begin{equation}\label{eq:logistic_DF_expression}
    (DF(\text{\textbf{x}}))\text{\textbf{y}} = \text{\textbf{y}} + \lambda(2u_0 - 1)P^T\text{\textbf{y}} + 2\lambda \mathcal{H}(\text{\textbf{x}}^T\text{\textbf{w}}(t)\text{\textbf{w}}^T(t)\text{\textbf{y}}).
\end{equation}
where the last equality comes from the symmetry of the bilinear form. 

Before proceeding, it is worth outlining the general strategy for the estimates. We separate the operator matrices according to the finite-infinite decomposition in \eqref{eq:finite_infinite_decomposition}. Then, all the finite-dimensional parts are collected together and left for the computer to calculate, while we use the analytic estimates from Section~\ref{sect:nonlinear} for the infinite parts. 

\begin{itemize}

    \item $Y_M$: Using the decomposition of $\lambda\mathcal{H}\left(\Bar{\text{\textbf{c}}}^T\text{\textbf{w}}(t)\text{\textbf{w}}^T(t)\Bar{\text{\textbf{c}}} \right)$ and Proposition~\ref{prop:estimates_finite_proj_finite_vectors},
\end{itemize}
\begin{align*}
    \Pi_M(T(\Bar{\text{\textbf{c}}}) - \Bar{\text{\textbf{c}}}) &= - A_M \Pi_M\mathcal{H}F(\Bar{\text{\textbf{c}}}) \\
    &= -A_M \Big[ \lambda(u_0^2-u_0)\text{\textbf{e}} + \Bar{\text{\textbf{c}}} + \lambda(2u_0-1) P_M^T\Bar{\text{\textbf{c}}} + 2\lambda \big( (\Tilde{\Omega}_M^TP_M^T\Bar{\text{\textbf{c}}}) \odot (P_M^T\Bar{\text{\textbf{c}}}) \big) \\
    &\hspace{4em}+ \lambda\Gamma^T_M (\Bar{\text{\normalfont{\textbf{a}}}}_M \odot \Bar{\text{\normalfont{\textbf{a}}}}_M) \Big] - \lambda A_M\Pi_M \Gamma^T_{\infty,M} (\Bar{\text{\normalfont{\textbf{a}}}}_\infty \odot \Bar{\text{\normalfont{\textbf{a}}}}_\infty),
\end{align*} 
with $\Bar{\text{\normalfont{\textbf{a}}}} = P^T\Bar{\text{\normalfont{\textbf{c}}}}$. Note that the term in the brackets can be computationally evaluated. Using the bound from Proposition~\ref{prop:estimates_finite_proj_finite_vectors} (iii) for the last term, we can define $Y_M$ as
\begin{equation}\label{eq:logistic_YM_expression}
    \begin{aligned}
        Y_M &:= \Big\| A_M \Big[ \lambda(u_0^2 - u_0)\text{\textbf{e}} + \Bar{\text{\textbf{c}}} + \lambda(2u_0 - 1)P_M^T\Bar{\text{\textbf{c}}} + 2\lambda \big( (\Tilde{\Omega}_M^TP_M^T\Bar{\text{\textbf{c}}}) \odot (P_M^T\Bar{\text{\textbf{c}}}) \big) + \lambda\Gamma^T_M (\Bar{\text{\normalfont{\textbf{a}}}}_M \odot \Bar{\text{\normalfont{\textbf{a}}}}_M) \Big] \Big\|_{\ell^2} \\
        &\hspace{2em}+ \frac{|\lambda| \|A_M\| \sqrt{2}}{\left(4-\sqrt{2}\right)2^{\frac{3J}{2}+3}} \|\Bar{\text{\textbf{c}}}\|_{\ell^2}^2
    \end{aligned}
\end{equation}

\begin{itemize}
    \item $Y_\infty$: Observe that 
\end{itemize}
\begin{equation*}
    \Pi_\infty(T(\Bar{\text{\textbf{c}}}) - \Bar{\text{\textbf{c}}}) = \Pi_\infty \left( \lambda(2u_0-1)P^T\Bar{\text{\textbf{c}}} + \lambda\mathcal{H}(\Bar{\text{\textbf{a}}}^T\Omega(t)\Bar{\text{\textbf{a}}} + \Bar{\text{\textbf{a}}}^T\Omega^T(t)\Bar{\text{\textbf{a}}} + \Bar{\text{\textbf{a}}}^T\Theta(t)\Bar{\text{\textbf{a}}} \right).
\end{equation*}
With the estimates from Proposition~\ref{prop:estimates_infinite_proj_finite_vectors}, we can make $Y_\infty$ as
\begin{equation}\label{eq:logistic_Yinf_expression}
    Y_\infty := \frac{\left|\lambda(2u_0-1)\right|}{\sqrt{3}} \frac{\left\|\Bar{\text{\textbf{c}}}\right\|_{\ell^2}}{2^{J+2}} + \frac{|\lambda|}{\sqrt{3}} \frac{\left\|\Bar{\text{\textbf{c}}}\right\|_{\ell^2}^2}{2^{2J+3}} +  \frac{|\lambda|}{21\sqrt{7}}\frac{\left\|\Bar{\text{\textbf{c}}}\right\|_{\ell^2}^2}{2^{3J+6}}.
\end{equation}

\begin{itemize}
    \item $Z_M$: Using the finite-infinite decomposition and the fact that $\Bar{\text{\textbf{c}}} \in \R^M$, 
\end{itemize}
\begin{align*}
    \Pi_M(DT(\Bar{\text{\textbf{c}}} + \text{\textbf{x}}))\text{\textbf{y}} &= \left( I_M - A_M B_1\right) \text{\textbf{y}}_M - A_M B_2 P_{\infty,M}^T \text{\textbf{y}}_\infty  - 2\lambda A_M\text{diag}(P_M^T\Bar{\text{\textbf{c}}})\big(\Tilde{\Omega}^TP^T\big)_{M,\infty}\text{\textbf{y}}_\infty \\
    &\hspace{3em}- 2\lambda A_M \Gamma_{\infty,M}^T\Pi_\infty \left( P^T \Bar{\text{\textbf{c}}} \odot P^T \text{\textbf{y}} \right) - 2\lambda A_M\Pi_M\mathcal{H}(\text{\textbf{x}}^T\text{\textbf{w}}(t)\text{\textbf{w}}^T(t)\text{\textbf{y}})
\end{align*}
where
\begin{align*}
    B_1 &:= I_M + \lambda(2u_0-1)P_M^T + 2\lambda \text{diag} \left( \Bar{\text{\textbf{c}}}^T P_M\Tilde{\Omega}_M \right)P_M^T \\
    &\hspace{5em}+ 2\lambda \text{diag}(P_M^T\Bar{\text{\textbf{c}}})\Tilde{\Omega}_M^TP_M^T + 2\lambda \Gamma_M^T \text{diag}\left(P_M^T\Bar{\text{\textbf{c}}} \right) P_M^T \\
    B_2 &:= \lambda(2u_0-1)I_M + 2\lambda \text{diag}\left( \Bar{\text{\textbf{c}}}^TP_M\Tilde{\Omega}_M \right) + 2\lambda \text{diag}(P_M^T\Bar{\text{\textbf{c}}}) +  2\lambda\Gamma_M^T\text{diag}\left(P_M^T\Bar{\text{\textbf{c}}}\right) 
\end{align*}
and $I_M$ is the $M \times M$ identity matrix. While the expression of $B_1$ and $B_2$ seem complicated, all terms are finite-dimensional and hence their norms can be calculated computationally. Thus, using the estimates from Proposition~\ref{prop:estimates_finite_proj_infinite_vectors} with the expression for $DT$ to bound the terms which cannot be easily estimated computationally, 
\begin{equation*}
    \left\| \Pi_M(DT(\Bar{\text{\textbf{c}}} + \text{\textbf{x}}))\text{\textbf{y}} \right\|_{\ell^2} \leq \left( C_1 + C_2 \left\|\text{\textbf{x}}\right\|_{\ell^2} \right) \left\|\text{\textbf{y}}\right\|_{\ell^2}
\end{equation*}
where
\begin{align*}
    C_1 &:= \left\|I_M - A_M B_1\right\| + \frac{\left\|A_MB_2\right\|}{2^{J+2}} 
    + \frac{(1+\sqrt{2})\left\|\lambda A_M\text{diag}(P_M^T\Bar{\text{\textbf{c}}})\right\|}{2^\frac{J+1}{2}} + \frac{\sqrt{2} \, \|\lambda A_M\|\,\|\Bar{\text{\textbf{c}}}\|_{\ell^2} }{(4-\sqrt{2})\,2^{\frac{3J}{2}+2}}\\
    C_2 &:= 2 |\lambda| 
    \left(2K_1 + K_2 \right)
\end{align*}
and $K_1$ and $K_2$ are as in Proposition~\ref{prop:estimates_finite_proj_infinite_vectors}. Hence, we can make $Z_M(r)$ as
\begin{equation}\label{eq:logistic_ZM_expression}
    Z_M(r) := C_1 + C_2r
\end{equation}

\begin{itemize}
    \item $Z_\infty$: We have that 
\end{itemize}
\begin{equation*}
    \Pi_\infty (DT(\Bar{\text{\textbf{c}}} + \text{\textbf{x}}))\text{\textbf{y}} =  \lambda(2u_0-1)\Pi_\infty P^T\text{\textbf{y}} + 2\lambda \Pi_\infty\mathcal{H}(\Bar{\text{\textbf{c}}}^T\text{\textbf{w}}(t)\text{\textbf{w}}^T(t)\text{\textbf{y}}) + 2\lambda \Pi_\infty\mathcal{H}(\text{\textbf{x}}^T\text{\textbf{w}}(t)\text{\textbf{w}}^T(t)\text{\textbf{y}})
\end{equation*}
All terms are infinite-dimensional and need to be analitically estimated. Using the estimates from Propositions \ref{prop:estimates_infinite_proj_infinite_vectors} and \ref{prop:PTc_etimates}, we can make $Z_\infty(r)$ as 
\begin{equation}\label{eq:logistic_Zinf_expression}
    Z_\infty(r) := \left( \frac{\left|\lambda(2u_0-1)\right|}{2^{J+1}\sqrt{3}} + \frac{|\lambda|D_1\left\|\Bar{\text{\textbf{c}}}\right\|_{\ell^2}}{2^{\frac{3J}{2}}} \right) + \frac{|\lambda|D_2}{2^{\frac{3J}{2}+2}}\,r
\end{equation}
with $D_1$ and $D_2$ as in Proposition~\ref{prop:estimates_infinite_proj_infinite_vectors}.

\subsubsection{Results}

Figure~\ref{fig:logistic_solutions} shows the numerical solutions for $J=6$ and $J=10$ compared to the true solution, using $\lambda = 6$ and $u_0 = 0.2$ for both cases. Visually, the numerical solutions agrees with the true one.

\begin{figure}[h]
    \centering
    \begin{subfigure}{0.4\linewidth}
        \includegraphics[width=\linewidth]{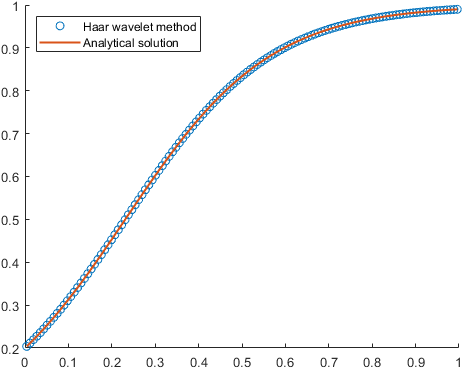}
        \caption{$J=6$}
        \label{fig:logistic_}
    \end{subfigure}
    \hfill
    \begin{subfigure}{0.4\linewidth}
        \includegraphics[width=\linewidth]{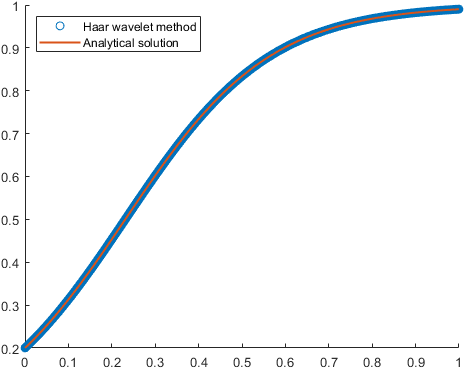}
        \caption{$J=10$}
        \label{fig:logistic_J10}
    \end{subfigure}
    \caption{Numerical and true solutions for the logistic equation.}
    \label{fig:logistic_solutions}
\end{figure}

Figure~\ref{fig:R0_vs_omega} shows the radius $r_0$ obtained as $J$ increases for different $\omega$. It is clear that, as $J$ increases, the radius $r_0$ decreases; this is due to more terms being calculated more accurately, instead of only being bounded by analytical estimates. Also, smaller values of $\omega$ yield tighter radii; however, if $\omega$ is too small the method will not work, as there will not be a true solution within $\overline{B_\omega(\Bar{\text{\textbf{c}}},r)}$. Figure~\ref{fig:R0_minimum_omega} shows the radii obtained with $\omega$ optimized up to two significant digits.

\begin{figure}[h!]
    \centering
    \includegraphics[width=0.6\linewidth]{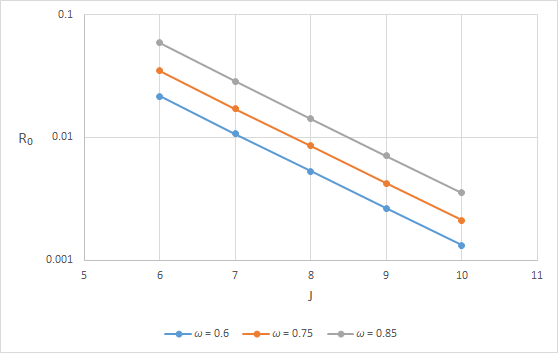}
    \begingroup
        \scriptsize
        \begin{tabular}{|c|c|c|c|}
            \hline
            $J$ & $r_0$ ($\omega = 0.6$) & $r_0$ ($\omega = 0.75$) & $r_0$ ($\omega = 0.85$) \\
            \hline
            6 & $2.1677704 \times 10^{-2}$ & $3.4976922 \times 10^{-2}$ & $5.9222878 \times 10^{-2}$ \\
            7 & $1.0690405 \times 10^{-2}$ & $1.7163569 \times 10^{-2}$ & $2.8785441 \times 10^{-2}$ \\
            8 & $5.3120948 \times 10^{-3}$ & $8.5120420 \times 10^{-3}$ & $1.4224769 \times 10^{-2}$ \\
            9 & $2.6483940 \times 10^{-3}$ & $4.2402808 \times 10^{-3}$ & $7.0756199 \times 10^{-3}$ \\
            10 & $1.3223867 \times 10^{-3}$ & $2.1164777 \times 10^{-3}$ & $3.5294192 \times 10^{-3}$ \\
            \hline
        \end{tabular}
    \endgroup
    \caption{Radius $r_0$ obtained for different values of $\omega$}
    \label{fig:R0_vs_omega}
\end{figure}

Lastly, Figure~\ref{fig:R0_vs_time} shows the radii and computation time. As $J$ increases, the computation time is expected to increase; however, Figure~\ref{fig:R0_vs_time} shows that after a certain point it increases more rapidly than $r_0$ decreases. This is expected as the size of the matrices quadruples for every increase of $J$, and so one must carefully balance the needed precision with computing time.

\begin{figure}[h!]
    \centering
    \begin{subfigure}{0.48\linewidth}
        \includegraphics[width=\linewidth]{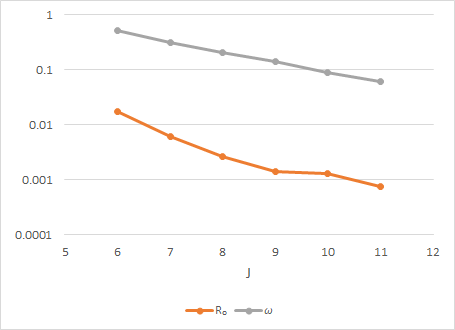}
        \caption{}
        \label{fig:R0_minimum_omega}
    \end{subfigure}
    \begin{subfigure}{0.50\linewidth}
        \includegraphics[width=\linewidth]{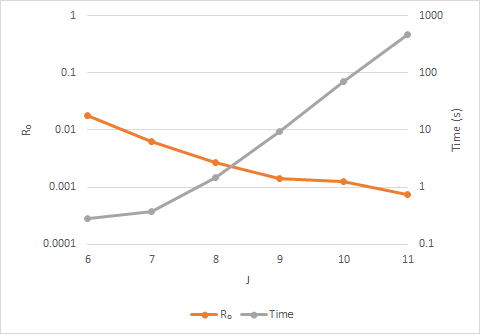}
        \caption{}
        \label{fig:R0_vs_time}
    \end{subfigure}
    \begingroup
        \scriptsize
        \begin{tabular}{|c|c|c|c|}
            \hline
            $J$ & $\omega$ & $r_0$ & Time \\
            \hline
            6 & 0.51 & $1.7651807 \times 10^{-2}$ & $0.283$ \\
            7 & 0.31 & $6.1826325 \times 10^{-3}$ & $0.377$ \\
            8 & 0.21 & $2.6863885 \times 10^{-3}$ & $1.499$ \\
            9 & 0.14 & $1.4138969 \times 10^{-3}$ & $9.447$ \\
            10 & 0.089 & $1.2789917 \times 10^{-3}$ & $70.958$ \\
            11 & 0.060 & $7.4420035 \times 10^{-4}$ & $474.220$ \\
            \hline
        \end{tabular}
    \endgroup
    \caption{(a) Radius $r_0$ obtained with more optimized $\omega$ for each resolution level $J$. (b) Comparison between $r_0$ and time elapsed.}
    \label{fig:logistic_optimal_radius}
\end{figure}

%

\subsection{Logistic equation with a discontinuous forcing term}
\label{subsect:forced}

The next example is again the logistic equation, but with a discontinuous forcing term
\begin{equation}\label{eq:forced_logistic_equation}
    \begin{aligned}
        & \dot{u} = \lambda u(1 - u) + g(t)\\
        & u(0) = u_0
    \end{aligned}
    \quad,\quad g(t) = 
    \begin{cases}
        1 &, \text{if } t \leq \frac{1}{2} \\
        0 &, \text{if } t > \frac{1}{2} \\
    \end{cases}
\end{equation} 
While this is a Ricatti equation which can be explicitly solved, we can see from the equation itself that the solution should not be smooth, since $g$ is discontinuous. Nonetheless, we can find a verification radius in the $L^2$ sense.

The functional equation for this case is similar to \eqref{eq:logistic_equation}; making $\text{\textbf{g}} = \mathcal{H}(g)$,
\begin{equation}\label{eq:forced_inf_matrix_form}
    \text{\textbf{c}}^T\text{\textbf{h}}(t) - \lambda\left(\text{\textbf{c}}^T\text{\textbf{w}}(t) + u_0 - u_0^2 - 2u_0 \text{\textbf{c}}^T\text{\textbf{w}}(t) - (W(t))(\text{\textbf{c}},\text{\textbf{c}})\right) - \text{\textbf{g}}^T\text{\textbf{h}}(t) = 0.
\end{equation}

\subsubsection{Numerical approximation}

Applying the same method as in the previous example, we get the equation to obtain the numerical, finite-dimensional approximation $\Bar{\text{\textbf{c}}}$, which is similar to before:
\begin{equation}\label{eq:forced_numerical_equation}
    \text{\textbf{c}}_M + \lambda(2u_0-1)P_M^T\text{\textbf{c}}_M + \frac{\lambda(u_0^2-u_0)}{M}H_M\text{\textbf{e}} + \frac{\lambda}{M}H_M\text{\textbf{W}}_M(\text{\textbf{c}}_M,\text{\textbf{c}}_M) - \frac{1}{M}H_M\text{\textbf{g}}_M = 0 
\end{equation}
where $\text{\textbf{g}}_M := (g(t_1)\,,\,...\,,\,g(t_M))^T$. Again, since this is a nonlinear equation, we will use Newton's method.

\subsubsection{Estimates for the radii polynomials}

In order to apply the radii polynomial method, we use the following functional equation $F:\ell^2(\R) \to \ell^2(\R)$:
\begin{equation*}
    F(\text{\textbf{x}}) := \lambda(u_0^2 - u_0)e_1 + \text{\textbf{x}} + \lambda(2u_0 - 1) P^T\text{\textbf{x}} + \lambda \mathcal{H}\left(\text{\textbf{x}}^T\text{\textbf{w}}(t)\text{\textbf{w}}^T(t)\text{\textbf{x}}\right) - \mathcal{H}(g).
\end{equation*}
Observe that $g = \frac{1}{2}(\psi_1 + \psi_2)$, and thus $\mathcal{H}(g) = \left(\frac{1}{2}, \frac{1}{2}, 0,0,...\right)^T =: \text{\textbf{g}}$; in particular, $\text{\textbf{g}} \in \R^M$. Hence, the map used is given by
\begin{equation}\label{eq:forced_F_expression}
    F(\text{\textbf{x}}) = \lambda(u_0^2 - u_0)e_1 + \text{\textbf{x}} + \lambda(2u_0 - 1) P^T\text{\textbf{x}} + \lambda \mathcal{H}\left(\text{\textbf{x}}^T\text{\textbf{w}}(t)\text{\textbf{w}}^T(t)\text{\textbf{x}}\right) - \text{\textbf{g}}.
\end{equation}

The maps $T$ and $DT(\text{\textbf{x}})$ for the radii polynomials are the same as \eqref{eq:T_fixed_point} and \eqref{eq:DT_fixed_point}, respectively. Actually, since $\text{\textbf{g}}$ does not depend on $\text{\textbf{x}}$, the derivative $DF$ is the same as in the non-forced logistic equation from \eqref{eq:logistic_DF_expression}.

Using the same methods as before, we have the following bounds for the radii polynomial method:
\begin{align}
    Y_M &:= C_0 + \frac{|\lambda|\|A_M\|\sqrt{2}}
        {\left(4-\sqrt{2}\right)2^{\frac{3J}{2}+3}} \|\Bar{\text{\textbf{c}}}\|_{\ell^2}^2 \label{eq:forced_YM_expression} \\
    Y_\infty &:= \frac{\left|\lambda(2u_0-1)\right|}{\sqrt{3}}
        \frac{\left\|\Bar{\text{\textbf{c}}}\right\|_{\ell^2}}{2^{J+2}} + \frac{|\lambda|}{\sqrt{3}} \frac{\left\|\Bar{\text{\textbf{c}}}\right\|_{\ell^2}^2}{2^{2J+3}} +  \frac{|\lambda|}{21\sqrt{7}}\frac{\left\|\Bar{\text{\textbf{c}}}\right\|_{\ell^2}^2}{2^{3J+6}} \label{eq:forced_Yinf_expression} \\
    Z_M(r) &:= C_1 + C_2r \label{eq:forced_ZM_expression} \\
    Z_\infty(r) &:= \left( \frac{\left|\lambda(2u_0-1)\right|}{2^{J+1}\sqrt{3}} +
        \frac{|\lambda|D_1\left\|\Bar{\text{\textbf{c}}}\right\|_{\ell^2}}{2^{\frac{3J}{2}}} \right) + \frac{|\lambda|D_2}{2^{\frac{3J}{2}+2}}\,r \label{eq:forced_Zinf_expression}
\end{align}
where $C_1$, $C_2$, $D_1$ and $D_2$ are as in the estimates for the non-forced logistic equation, and
\begin{align*}
    C_0 &= \Big\| A_M \Big[ \lambda(u_0^2 - u_0)\text{\textbf{e}} + \Bar{\text{\textbf{c}}} + \lambda(2u_0 - 1)P_M^T\Bar{\text{\textbf{c}}} \\
    &\hspace{5em}+ 2\lambda \Big( (\Tilde{\Omega}_M^TP_M^T\Bar{\text{\textbf{c}}}) \odot (P_M^T\Bar{\text{\textbf{c}}}) \Big) + \lambda\Gamma^T_M (\Bar{\text{\normalfont{\textbf{a}}}}_M \odot \Bar{\text{\normalfont{\textbf{a}}}}_M) - \text{\textbf{g}} \Big] \Big\|_{\ell^2}.
\end{align*}

\subsubsection{Results}

Figure \ref{fig:forced_solutions} shows the results using the Haar wavelet method compared to numerical integration, using $\lambda = 6$ and $u_0 = 0.2$. For the numerical integration, we used the same amount of points as the Haar wavelet method, that is, $2^{J+1}$ points. It can be seen that the numerical integration tends to smooth the graph at 
$t=0.5$, while our method preserves the original shape. 

\begin{figure}[h]
    \centering
    \begin{subfigure}{0.4\linewidth}
        \includegraphics[width=\linewidth]{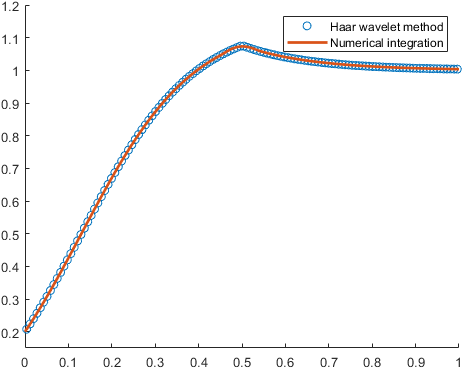}
        \caption{$J=6$}
        \label{fig:forced_solution_J6}
    \end{subfigure}
    \hfill
    \begin{subfigure}{0.4\linewidth}
        \includegraphics[width=\linewidth]{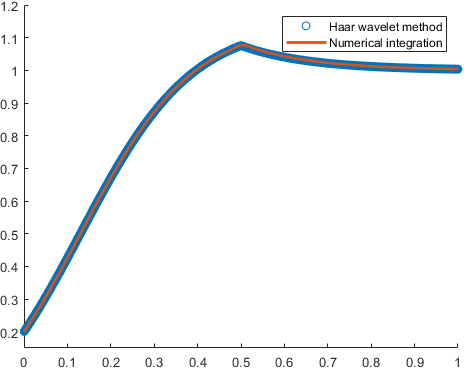}
        \caption{$J=10$}
        \label{fig:forced_solution_J10}
    \end{subfigure}
    \caption{Numerical and true solutions for the logistic equation with forcing term}
    \label{fig:forced_solutions}
\end{figure}

Figure \ref{fig:forced_R0} shows the verification radius as $J$ increases. It is worth noting that the solution is not smooth; nonetheless, our method returned verification radii similar to the non-forced logistic equation. 

\begin{figure}[h!]
    \centering
    \hfill
    \begin{minipage}[c]{0.4\linewidth}
        \vspace{0pt}
        \includegraphics[width=\linewidth]{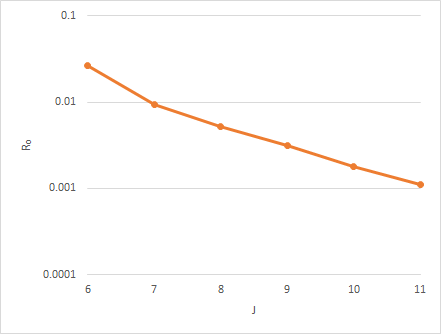}
    \end{minipage}
    \hfill
    \begin{minipage}[c]{0.4\linewidth}
        \scriptsize
        \vspace{0pt}
        \begin{tabular}{|c|c|c|}
            \hline
            $J$ & $\omega$ & $r_0$ \\
            \hline
            6 & 0.53  & $2.6161420 \times 10^{-2}$ \\
            7 & 0.31  & $9.2508029 \times 10^{-3}$ \\
            8 & 0.20  & $5.1495598 \times 10^{-3}$ \\
            9 & 0.13  & $3.1382945 \times 10^{-3}$ \\
            10 & 0.086 & $1.7710909 \times 10^{-3}$ \\
            11 & 0.057 & $1.1107730 \times 10^{-3}$ \\
            \hline
        \end{tabular}
    \end{minipage}
    \caption{Verification radius for the forced logistic equation}
    \label{fig:forced_R0}
\end{figure}

\subsection{Lorenz system}\label{subsect:lorenz}

The Lorenz system is given by
\begin{equation}\label{eq:lorenz_system}
    \begin{aligned}
        \dot{x} &= \sigma(y-x) \\
        \dot{y} &= x(\rho-z)-y \\
        \dot{z} &= xy-\beta z
    \end{aligned}
\end{equation}
where $\sigma$, $\rho$ and $\beta$ are positive parameters, usually taken as $\sigma = 10$, $\beta = \frac{8}{3}$ and $\rho = 28$. This is a well-studied system, and with these parameters the system exhibits chaotic behavior with a strange attractor.

Since this is a system of equations, we must first some of our definitions in order to apply our method. First, we define the spaces $X_s := \ell^2(\R) \times \ell^2(\R) \times \ell^2(\R)$ and $X_f := L^2([0,1]) \times L^2([0,1]) \times L^2([0,1])$, and endow them with the norms
\begin{equation}\label{eq:norm_system}
    \begin{aligned}
        \|(\text{\textbf{c}}_1,\text{\textbf{c}}_2,\text{\textbf{c}}_3)\|_{X_s} &:= \max \left\{ \|\text{\textbf{c}}_1\|_{\ell^2}, \|\text{\textbf{c}}_2\|_{\ell^2}, \|\text{\textbf{c}}_3\|_{\ell^2} \right\} \\
        \|(f_1,f_2,f_3)\|_{X_f} &:= \max \left\{ \|f_1\|_{L^2}, \|f_2\|_{L^2}, \|f_3\|_{L^2} \right\}
    \end{aligned}    
\end{equation}
With these norms, $X_s$ and $X_f$ are still Banach spaces; though they are no longer Hilbert spaces, the methods of Section \ref{sect:radii_poly} are still applicable. Also, an operator $A: X_s \to X_s$ can be expressed using block matrix notation as
\begin{equation*}
    A = 
    \begin{bmatrix}
        A_{1,1} & A_{1,2} & A_{1,3} \\
        A_{2,1} & A_{2,2} & A_{2,3} \\
        A_{3,1} & A_{3,2} & A_{3,3} 
    \end{bmatrix}
\end{equation*}
where $A_{i,j}: \ell^2(\R) \to \ell^2(\R)$ for $i,j = 1,2,3$. If those are bounded, $A$ is bounded and
\begin{equation}\label{eq:norm_operator}
    \|A\|_{B(X_s)} = \max_{1 \leq i \leq 3} \sum_{j=1}^3 \|A_{i,j}\|_{B(\ell^2)}.
\end{equation}
Similar notation will be used when $A_{i,j} \in B(\R^M)$. Lastly, we make a small abuse of notation and extend the notation for the operators in $\ell^2(\R)$ such as the projections $\Pi_M$ and $\Pi_\infty$ to $X_s$ by applying them element-wise:
\begin{align*}
    \Pi_M(\text{\textbf{c}}_1, \text{\textbf{c}}_2, \text{\textbf{c}}_3) &:= (\Pi_M\text{\textbf{c}}_1, \Pi_M\text{\textbf{c}}_2, \Pi_M\text{\textbf{c}}_3) \\
    \Pi_\infty(\text{\textbf{c}}_1, \text{\textbf{c}}_2, \text{\textbf{c}}_3) &:= (\Pi_\infty\text{\textbf{c}}_1, \Pi_\infty\text{\textbf{c}}_2, \Pi_\infty\text{\textbf{c}}_3).
\end{align*}

Applying the same methods used to obtain \eqref{eq:logistic_inf_matrix_form} to each equation in \eqref{eq:lorenz_system}, we obtain the functional equation
\begin{equation}\label{eq:lorenz_inf_matrix_form}
    \begin{aligned}
        & \text{\textbf{c}}_x^T\text{\textbf{h}}(t) - \sigma\text{\textbf{c}}_y^T\text{\textbf{w}}(t) + \sigma\text{\textbf{c}}_x^T\text{\textbf{w}}(t) -\sigma(y_0-x_0) = 0 \\
        & \text{\textbf{c}}_y^T\text{\textbf{h}}(t) + x_0\text{\textbf{c}}_z^T\text{\textbf{w}}(t) - (\rho-z_0)\text{\textbf{c}}_x^T\text{\textbf{w}}(t) + \text{\textbf{c}}_y^T\text{\textbf{w}}(t) - (x_0(\rho-z_0)-y_0) + W(t)(\text{\textbf{c}}_x,\text{\textbf{c}}_z) = 0 \\
        & \text{\textbf{c}}_z^T\text{\textbf{h}}(t) - y_0\text{\textbf{c}}_x^T\text{\textbf{w}}(t) - x_0\text{\textbf{c}}_y^T\text{\textbf{w}}(t) + \beta \text{\textbf{c}}_z^T\text{\textbf{w}}(t) - (x_0y_0-\beta z_0) - W(t)(\text{\textbf{c}}_x,\text{\textbf{c}}_y) = 0
    \end{aligned}
\end{equation}
where $\text{\textbf{c}}_x = \mathcal{H}(\dot{x})$, $\text{\textbf{c}}_y = \mathcal{H}(\dot{y})$ and $\text{\textbf{c}}_z = \mathcal{H}(\dot{z})$.

\subsubsection{Numerical approximation}

Using the same techniques used to obtain \eqref{eq:logistic_numerical_equation} and \eqref{eq:forced_numerical_equation} to each equation in \eqref{eq:lorenz_inf_matrix_form}, we obtain the system to be solved numerically with Newton's method:
\begin{equation}\label{eq:lorenz_numerical_equation}
    \begin{aligned}
        & \Bar{\text{\textbf{c}}}_x - \frac{\sigma}{M} P_M^T\Bar{\text{\textbf{c}}}_y + \frac{\sigma}{M} P_M^T\Bar{\text{\textbf{c}}}_x -\frac{\sigma}{M} (y_0-x_0)H_M\text{\textbf{e}} = 0 \\
        & \Bar{\text{\textbf{c}}}_y + \frac{1}{M}H_M \text{\textbf{W}}_M(\Bar{\text{\textbf{c}}}_x,\Bar{\text{\textbf{c}}}_z) + x_0 P_M^T\Bar{\text{\textbf{c}}}_z - (\rho-z_0) P_M^T\Bar{\text{\textbf{c}}}_x + P_M^T\Bar{\text{\textbf{c}}}_y - \frac{1}{M}(x_0(\rho-z_0)-y_0)H_M\text{\textbf{e}} = 0 \\
        & \Bar{\text{\textbf{c}}}_z - \frac{1}{M}H_M \text{\textbf{W}}_M(\Bar{\text{\textbf{c}}}_x,\Bar{\text{\textbf{c}}}_y) - y_0 P_M^T\Bar{\text{\textbf{c}}}_x - x_0 P_M^T\Bar{\text{\textbf{c}}}_y + \beta P_M^T\Bar{\text{\textbf{c}}}_z - \frac{1}{M}(x_0y_0-\beta z_0)H_M\text{\textbf{e}} = 0
    \end{aligned}
\end{equation}

\begin{remark}
    One interesting remark in \cite{2009lepik} is that one can use results from lower resolutions as initial guesses for the Newton's method for higher resolution levels, instead of using a high resolution level right from the start. For the Lorenz system, this can reduce convergence problems and overall calculation time.
\end{remark}

\subsubsection{Estimates for the radii polynomials}

For the Lorenz system, the maps $T$ and $DT$ are as in \eqref{eq:T_fixed_point} and \eqref{eq:DT_fixed_point} respectively, with $\text{\textbf{c}} = (\text{\textbf{c}}_x, \text{\textbf{c}}_y, \text{\textbf{c}}_z) \in X_s$. The map $F:X_s \to X_s$ is given by
\begin{equation*}
    F(\text{\textbf{c}}) := \left(
    \begin{array}{l}
        \text{\textbf{c}}_x - \sigma P^T\text{\textbf{c}}_y + \sigma P^T\text{\textbf{c}}_x - \sigma(y_0-x_0)\text{\textbf{e}}_1 \\[5pt]
        \text{\textbf{c}}_y + \mathcal{H}\left( \text{\textbf{c}}_x^T \text{\textbf{w}}(t) \text{\textbf{w}}^T(t) \text{\textbf{c}}_z \right) + x_0 P^T\text{\textbf{c}}_z  - (\rho-z_0) P^T\text{\textbf{c}}_x + P^T\text{\textbf{c}}_y - (x_0(\rho-z_0)-y_0)\text{\textbf{e}}_1 \\[5pt]
        \text{\textbf{c}}_z - \mathcal{H}\left( \text{\textbf{c}}_x^T \text{\textbf{w}}(t) \text{\textbf{w}}^T(t) \text{\textbf{c}}_y \right) - y_0 P^T\text{\textbf{c}}_x - x_0 P^T\text{\textbf{c}}_y + \beta P^T\text{\textbf{c}}_z - (x_0y_0-\beta z_0)\text{\textbf{e}}_1
    \end{array}
    \right)
\end{equation*}
and for $\text{\textbf{v}} = (\text{\textbf{v}}_x, \text{\textbf{v}}_y, \text{\textbf{v}}_z) \in X_s$, 
\begin{equation*}
    DF(\text{\textbf{c}})\text{\textbf{v}} := \left(
    \begin{array}{l}
        \text{\textbf{v}}_x - \sigma P^T\text{\textbf{v}}_y + \sigma P^T\text{\textbf{v}}_x \\[5pt]
        \text{\textbf{v}}_y + \mathcal{H}\left( \text{\textbf{c}}_x^T \text{\textbf{w}}(t) \text{\textbf{w}}^T(t) \text{\textbf{v}}_z \right) + \mathcal{H}\left( \text{\textbf{c}}_z^T \text{\textbf{w}}(t) \text{\textbf{w}}^T(t) \text{\textbf{v}}_x \right) \\[2pt]
        \hspace{7em}+ x_0 P^T\text{\textbf{v}}_z - (\rho-z_0) P^T\text{\textbf{v}}_x + P^T\text{\textbf{v}}_y \\[5pt]
        \text{\textbf{c}}_z - \mathcal{H}\left( \text{\textbf{c}}_x^T \text{\textbf{w}}(t) \text{\textbf{w}}^T(t) \text{\textbf{v}}_y \right) - \mathcal{H}\left( \text{\textbf{c}}_y^T \text{\textbf{w}}(t) \text{\textbf{w}}^T(t) \text{\textbf{v}}_x \right) \\[2pt]
        \hspace{7em}- y_0 P^T\text{\textbf{v}}_x - x_0 P^T\text{\textbf{v}}_y + \beta P^T\text{\textbf{v}}_z
    \end{array}
    \right).
\end{equation*}

For the Lorenz system, we have the following estimates for the radii polynomial method. Using the norms defined in \eqref{eq:norm_system} and \eqref{eq:norm_operator} and applying the same techniques and techniques as before to each equation, we find the following bounds:
\begin{equation}\label{eq:lorenz_estimates}
    \begin{aligned}
        Y_M &= \|A_MF_M(\Bar{\text{\textbf{c}}})\|_{X_s} + \frac{\|A_M\|_{B(X_s)} \|\Bar{\text{\textbf{c}}}_x\|_{\ell^2}}{(4-\sqrt{2})2^{\frac{3J+5}{2}}} \max\big\{ \| \Bar{\text{\textbf{c}}}_y\|_{\ell^2} ,  \|\Bar{\text{\textbf{c}}}_z\|_{\ell^2} \big\} \\
        Y_\infty &= \max\{Y_\infty^1, Y_\infty^2, Y_\infty^3\} \\
        Z_M(r) &= \alpha_1 + \alpha_2r \\
        Z_\infty(r) &= \gamma_1 + \gamma_2r
    \end{aligned}
\end{equation}
where the constants above are given by
\begin{align*}
    Y_\infty^1 &:=
        \frac{|\sigma|}{\sqrt{3}} \frac{\|\Bar{\text{\textbf{c}}}_x 
        - \Bar{\text{\textbf{c}}}_y\|_{\ell^2}}{2^{J+2}} \\
    Y_\infty^2 &:=
        \frac{\|\Bar{\text{\textbf{c}}}_x\|_{\ell^2} \|\Bar{\text{\textbf{c}}}_z\|_{\ell^2}}{\sqrt{3}\ 2^{2J+5}} 
        + \frac{\|\Bar{\text{\textbf{c}}}_x\|_{\ell^2} \|\Bar{\text{\textbf{c}}}_z\|_{\ell^2}}{(21\sqrt{7})\,2^{3J+6}} 
        + \frac{\|x_0\Bar{\text{\textbf{c}}}_z - (\rho-z_0)\Bar{\text{\textbf{c}}}_x + \Bar{\text{\textbf{c}}}_y\|_{\ell^2}}{\sqrt{3}\ 2^{J+2}} \\
    Y_\infty^3 &:=
        \frac{\|\Bar{\text{\textbf{c}}}_x\|_{\ell^2} \|\Bar{\text{\textbf{c}}}_y\|_{\ell^2}}{2^{2J+5}\ \sqrt{3}} 
        + \frac{\|\Bar{\text{\textbf{c}}}_x\|_{\ell^2} \|\Bar{\text{\textbf{c}}}_y\|_{\ell^2}}{(21\sqrt{7})\,2^{3J+6}} 
        + \frac{\|\beta\Bar{\text{\textbf{c}}}_z - y_0\Bar{\text{\textbf{c}}}_x - x_0\Bar{\text{\textbf{c}}}_y\|_{\ell^2}}{\sqrt{3}\ 2^{J+2}}
\end{align*}
\begin{align*}
    \alpha_1 &:= \left\|I_M-A_MB_1\right\|_{B(X_s)} 
        + \frac{\left\|A_MB_2\right\|_{B(X_s)}}{2^{J+2}}  \\
        & \hspace{4em} + \frac{1+\sqrt{2}}{2^\frac{J+3}{2}}\left\|A_MB_2\right\|_{B(X_s)} + \frac{\sqrt{2}}{4-\sqrt{2}} \left\|A_M^\dagger
        \begin{pmatrix}
            0 \\
            \|\Bar{\text{\textbf{c}}}_x\|_{\ell^2} + \|\Bar{\text{\textbf{c}}}_z\|_{\ell^2} \\
            \|\Bar{\text{\textbf{c}}}_x\|_{\ell^2} + \|\Bar{\text{\textbf{c}}}_y\|_{\ell^2}
        \end{pmatrix}
        \right\|_{X_s} \\
    \alpha_2 &:= \|4C+2D\|_{B(X_s)}
\end{align*}
\begin{align*}
    \gamma_1 &:= \max \bigg\{ \frac{|\sigma|}{\sqrt{3}\ 2^{J+1}} \ , \ 
        \frac{D_1 \big(\|\Bar{\text{\textbf{c}}}_z\|_{\ell^2} + \|\Bar{\text{\textbf{c}}}_x\|_{\ell^2}\big)}{2^\frac{3J}{2}} + \frac{1 + |\rho-z_0| + |x_0|}{\sqrt{3}\ 2^{J+2}} \ , \\ 
        &\hspace{8em} \frac{D_1 \big(\|\Bar{\text{\textbf{c}}}_z\|_{\ell^2} + \|\Bar{\text{\textbf{c}}}_x\|_{\ell^2}\big)}{2^\frac{3J}{2}} + \frac{\beta + |x_0| + |y_0|}{\sqrt{3}\ 2^{J+2}} \bigg\} \\
    \gamma_2 &:= \frac{D_2}{2^{\frac{3J}{2}+1}}.
\end{align*}
and the auxiliary quantities to calculate the constants are given by
\begin{equation*}
    A_M^\dagger :=
        \begin{bmatrix}
            \|A_{M_{x,x}}\|_{B(\ell^2)} & \|A_{M_{x,y}}\|_{B(\ell^2)} & \|A_{M_{x,z}}\|_{B(\ell^2)} \\
            \|A_{M_{y,x}}\|_{B(\ell^2)} & \|A_{M_{y,y}}\|_{B(\ell^2)} & \|A_{M_{y,z}}\|_{B(\ell^2)} \\
            \|A_{M_{z,x}}\|_{B(\ell^2)} & \|A_{M_{z,y}}\|_{B(\ell^2)} & \|A_{M_{z,z}}\|_{B(\ell^2)}
        \end{bmatrix}
\end{equation*}
\begin{equation*}
    \begingroup
        \setlength\arraycolsep{4pt}
        B_1 := 
        \begin{bmatrix}
            I_M+\sigma P_M^T & -\sigma P_M^T & 0 \\
            (B_M(\Bar{\text{\textbf{c}}}_z) - (\rho-z_0)I_M)P_M^T & I_M+P_M^T & (B_M(\Bar{\text{\textbf{c}}}_x)+x_0I_M)P_M^T \\
            (-B_M(\Bar{\text{\textbf{c}}}_y) - y_0I_M)P_M^T & -(B_M(\Bar{\text{\textbf{c}}}_x) + x_0I_M)P_M^T & I_M+P_M^T 
        \end{bmatrix}
    \endgroup 
\end{equation*}
\begin{equation*}
    B_M(\Bar{\text{\textbf{c}}}_i) := \text{diag}(\Tilde{\Omega}_M^T P_M^T\Bar{\text{\textbf{c}}}_i) + \Gamma_M^T\text{diag}(P_M^T\Bar{\text{\textbf{c}}}_i)
\end{equation*}
\begin{equation*}
    B_2 := 
    \begin{bmatrix}
        \sigma I_M & -\sigma I_M & 0 \\
        B_M(\Bar{\text{\textbf{c}}}_z) & I_M & B_M(\Bar{\text{\textbf{c}}}_x) \\
        -B_M(\Bar{\text{\textbf{c}}}_y) & -B_M(\Bar{\text{\textbf{c}}}_x) & I_M
    \end{bmatrix} 
\end{equation*}
\begin{equation*}
    C = 
    \begin{bmatrix}
      0 & C_{x,y} & C_{x,z} \\
      0 & C_{y,y} & C_{y,z} \\
      0 & C_{z,y} & C_{z,z}
    \end{bmatrix} \quad,\quad 
    D = 
    \begin{bmatrix}
      0 & D_{x,y} & D_{x,z} \\
      0 & D_{y,y} & D_{y,z} \\
      0 & D_{z,y} & D_{z,z}
    \end{bmatrix}
\end{equation*}
and for $i,j\in\{x,y,z\}$
\begin{align*}
    C_{i,j} &:=  \| A_{M_{i,j}}\, \text{diag}\left(\|(P_M^T)_{i,*}\|_{\ell^2}\right) \Tilde{\Omega}_M^TP_M^T \| 
        + \frac{\left\|A_{M_{i,j}} \Tilde{\Omega}_M^T P_M^T\right\|}{2^{J+2}} \\
    & \hspace{8em} + \frac{(1+\sqrt{2})\left\|A_{M_{i,j}}P_M^T\right\|}{2^\frac{J+3}{2}}
        + \frac{(1+\sqrt{2})\left\|A_{M_{i,j}}\right\|}{2^\frac{3J+7}{2}} \\
    D_{i,j} &:= \left\| A_{M_{i,j}} \Gamma_{M_{i,j}}^T \text{diag}(\|(P_M^T)_{i,*}\|_\ell^2) P_M^T \right\| 
        + \frac{\left\| A_{M_{i,j}} \Gamma_{M_{i,j}}^T P_M^T \right\|}{2^{J+1}} \\
    & \hspace{8em} + \frac{\left\| A_{M_{i,j}}\Gamma_{M_{i,j}}^T \right\|}{2^{2J+4}}
        + \frac{\sqrt{2}}{\left(4-\sqrt{2}\right)}\, \frac{\|A_{M_{i,j}}\|}{2^{\frac{3J}{2}+4}}
\end{align*}
and $D_1$ and $D_2$ are as in Proposition~\ref{prop:estimates_infinite_proj_infinite_vectors}.

\subsubsection{Results}

Figure \ref{fig:lorenz_plot} shows the approximation obtained with the Haar wavelet method using $J = 10$ and numerical integration (fewer points from the Haar wavelet method are displayed for clarity); for the latter, we used $M^2 = 2^{2J+2}$ points in order to obtain a precise result. Visually, there seems to be good agreement between both results.
\begin{figure}[h]
    \centering
    \includegraphics[width=0.7\linewidth]{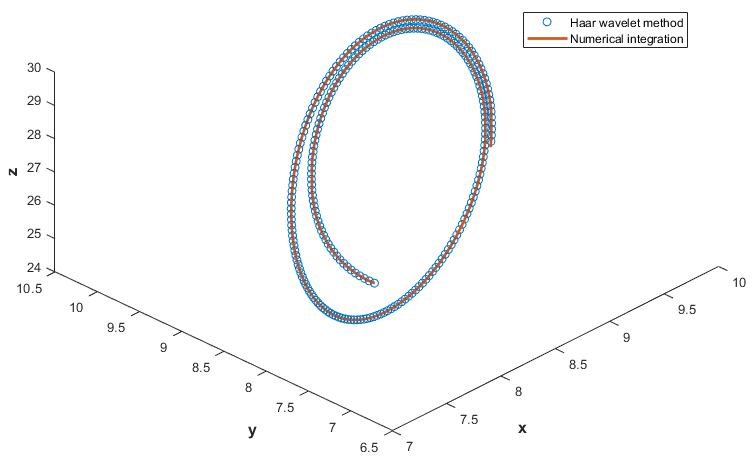}
    \caption{Results for the Lorenz system}
    \label{fig:lorenz_plot}
\end{figure}

For the same resolution level, the numerical result was rigorously verified by the radii polynomial method using the estimates described above, obtaining an $r_0 = 3.9868504 \times 10^{-2}$ for $\omega = 0.45$ in which the true solution lies in $X_s$.
\section{Conclusions and future work}
\label{sect:conclusion}

We developed a radii polynomial method using the Haar wavelet approach for differential equations, and illustrated the method by applying it to three differential equations. One advantage of our method over previous methods based on the radii polynomials approach is that, due to the use of the Haar wavelets, our method does not require the solutions to be smooth.

In the future we plan to develop the estimates for higher-order derivatives. While a higher-order differential equation can be transformed into a system of first-order equations, this increases the size of the matrices. Thus it might be interesting to use operators that directly represent higher-order derivatives. While some of those have already been used for usual numerical methods, we need to compute the estimates needed for our radii polynomial method.

Furthermore, we only presented the estimates needed for the radii polynomials for quadratic nonlinearities, since the main goal of this paper is to present the general method and illustrate how to compute the estimates and apply the method. In the future we plan to extend these estimates to include higher-order polynomial nonlinearities, as this would greatly expand the applicability of our method.

Lastly, we believe this method can be a basis to build other methods using similar techniques, such as continuation methods.

\subsection*{Acknowledgments}
The work of Guilherme K. Nakassima was supported by the Coordenação de Aperfeiçoamento de Pessoal de Nível Superior - Brasil (CAPES) - Finance Code 001. The work of Marcio Gameiro was partially supported by the National Science Foundation under awards DMS-1839294 and HDR TRIPODS award CCF-1934924, DARPA contract HR0011-16-2-0033, National Institutes of Health award R01 GM126555, and by CNPq grant 309073/2019-7.

\appendix
\section*{Appendices}
\addcontentsline{toc}{section}{Appendices}
\renewcommand{\thesubsection}{\Alph{subsection}}
\subsection{Proof of Theorem~\ref{thm:biform_haar_transform}}\label{apdx:biform}
Here we present the full proof of Theorem~\ref{thm:biform_haar_transform}. While lengthy, this proof not only validates the theorem that allows us to deal with quadratic nonlinearities, but also hints at how to find bounds for the radii polynomials. 

Let us outline the general strategy. We first make heuristical calculations to find an expression for the Haar transforms; then, by proving that they are indeed in $\ell^2(\R)$, the uniqueness of the Haar wavelet series justifies the calculations.

Let $\text{\textbf{c}},\text{\textbf{d}} \in \ell^2(\R)$, and assume that $\text{\textbf{c}}^T\text{\textbf{w}}(t)\text{\textbf{w}}^T(t)\text{\textbf{d}} \in L^2([0,1])$. We have
\begin{equation*}
    \text{\textbf{c}}^T\text{\textbf{w}}(t)\text{\textbf{w}}^T(t)\text{\textbf{d}} = \text{\textbf{a}}^T\text{\textbf{h}}(t)\text{\textbf{h}}^T(t)\text{\textbf{b}}
\end{equation*}
where we denote $\text{\textbf{a}} = P^T\text{\textbf{c}}$ and $\text{\textbf{b}} = P^T\text{\textbf{d}}$ for simplicity. Let us study the matrix $\text{\textbf{h}}(t)\text{\textbf{h}}^T(t)$ in more depth. First, by definition, we have $\psi_{j,k}(t)\psi_{j,q}(t) \not= 0$ if and only if $k = q$. For $\psi_{j,k}(t)\psi_{p,q}(t)$ when $j \not= p$, without loss of generality, consider $j < p$ for a fixed $p$. Because of the nesting property, $\psi_{j,k}$ is constant (possibly zero) in $\supp\psi_{p,q}$, so we can assume $\psi_{j,k}(t) = \psi_{j,k}\left(\frac{q+0.5}{2^p}\right)$ for $t \in \supp\psi_{p,q}$. Thus, we can write $\psi_{j,k}(t)\psi_{p,q}(t)$ as
\begin{equation*}
    \psi_{j,k}(t)\psi_{p,q}(t) = \psi_{j,k}\left(\frac{q+0.5}{2^p}\right)\psi_{p,q}(t) = (H_{m/2})_{i,l}\psi_{l}(t).
\end{equation*}
Hence, after adjusting indices and taking into account the symmetry of the matrix, we have proved the following:
\begin{lemma}\label{lem:h_hT_recursive_formula}
    The product $\text{\normalfont{\textbf{h}}}(t)\text{\normalfont{\textbf{h}}}^T(t)$ can be recursively calculated for $m = 2^j$, $j = 0,1,2,\dots$ as
    \begin{equation*}
        \text{\normalfont{\textbf{h}}}_1(t)\text{\normalfont{\textbf{h}}}_1^T(t) = \psi_1^2(t) = \phi^2(t) \quad,\quad 
        \text{\normalfont{\textbf{h}}}_{2m}(t) \text{\normalfont{\textbf{h}}}_{2m}^T(t) = 
        \begin{bmatrix}
            \text{\normalfont{\textbf{h}}}_{m}(t) \text{\normalfont{\textbf{h}}}_{m}^T(t) & \Upsilon_m(t) \\[0.1in]
            \Upsilon_m^T(t) & \Delta_m(t)
        \end{bmatrix}
    \end{equation*}
    where, for $i,l = 1,...,2^m$, $\Gamma_m$ and $\Delta_m$ are $m \times m$ matrices defined element-wise as
    \begin{equation*}
        (\Upsilon_m)_{i,l}(t) = (H_m)_{i,l}\psi_{m+l}(t) \quad,\quad (\Delta_m)_{i,l}(t) = 
        \begin{cases}
            \psi_i^2(t) &,\text{ if } i = l \\
            0 &,\text{ otherwise}.
        \end{cases}
    \end{equation*}
\end{lemma}
This justifies the decomposition of $\text{\normalfont{\textbf{h}}}(t)\text{\normalfont{\textbf{h}}}^T(t)$ as in Theorem~\ref{thm:biform_haar_transform}, that is, $\text{\textbf{h}}(t)\text{\textbf{h}}^T(t) = \Omega(t) + \Omega^T(t) + \Theta(t)$. 
We treat each term separately. First, for $\text{\textbf{a}}^T\Omega(t)\text{\textbf{b}}$, dividing $\text{\textbf{b}}$ in $2^j$ vector blocks, the product $\Omega(t)\text{\textbf{b}}$ is given element-wise by
\begin{align*}
    (\Omega(t)\text{\textbf{b}})_i &= \sum_{p=j}^\infty (\Upsilon_{2^p})_{i,*}\text{\textbf{b}}^*_{2^p}
    = \sum_{p=j}^\infty \sum_{q=1}^{2^p} (H_{2^p})_{i,q}\,\psi_{2^p+q}(t)\,b_{2^p+q} \\
    &= \sum_{p=j}^\infty (H_{2^p})_{i,*}\,(\text{\textbf{b}}_{2^p}^* \odot \text{\textbf{h}}_{2^p}^*(t)) =: (\Tilde{\Omega} (\text{\textbf{b}} \odot \text{\textbf{h}}(t)))_i
\end{align*}
where $\odot$ denotes the Hadamard (i.e. element-wise) product, and the matrix $\Tilde{\Omega}$ can be recursively constructed as
\begin{align*}
    \Tilde{\Omega}_1 = 0 \quad,\quad \Tilde{\Omega}_{2m} = 
    \begin{bmatrix}
        \Tilde{\Omega}_{m} & H_m \\
        0_m & 0_m
    \end{bmatrix}
	\quad, \text{ for } m = 2^j \text{ and } j = 0,1,2,...
\end{align*}
Thus, the product $\text{\textbf{a}}^T\Omega(t)\text{\textbf{b}}$ is given by
\begin{align*}
    \text{\textbf{a}}^T\Omega(t)\text{\textbf{b}} &= (\text{\textbf{c}}^T P) \Tilde{\Omega} \left(\text{\textbf{b}} \odot \text{\textbf{h}}(t)\right) = \sum_{i=1}^\infty \sum_{l=1}^\infty c_i(P\Tilde{\Omega})_{i,l}b_l\psi_l(t).
\end{align*}
In particular, if we assume $\text{\textbf{a}}^T\Omega(t)\text{\textbf{b}} \in L^2([0,1])$, then 
\begin{equation}\label{eq:haar_transform_biform_Omega_element}
    \left(\mathcal{H}\left(\text{\textbf{a}}^T\Omega(t)\text{\textbf{b}}\right)\right)_l = b_l \sum_{i=1}^\infty c_i(P\Tilde{\Omega})_{i,l} = (\text{\textbf{c}}^TP\Tilde{\Omega})_l b_l.
\end{equation}
or written in another way, $\mathcal{H}\left(\text{\textbf{a}}^T\Omega(t)\text{\textbf{b}}\right) = \big(\Tilde{\Omega}^T P^T \text{\textbf{c}}\big) \odot \text{\textbf{b}}$. 

Also, since the output of $\text{\textbf{a}}^T\Omega^T(t)\text{\textbf{b}}$ is a real number, then $\text{\textbf{a}}^T\Omega^T(t)\text{\textbf{b}} = \left(\text{\textbf{a}}^T\Omega^T(t)\text{\textbf{b}}\right)^T = \text{\textbf{b}}^T\Omega(t)\text{\textbf{a}}$,
and thus $\mathcal{H}\left(\text{\textbf{a}}^T\Omega^T(t)\text{\textbf{b}}\right) = \big(\Tilde{\Omega}^T P^T \text{\textbf{d}}\big) \odot \text{\textbf{a}}$ as well.

For the term $\mathcal{H}(\text{\textbf{a}}^T\Theta(t)\text{\textbf{b}})$, we have
\begin{equation*}
    \text{\textbf{a}}^T\Theta(t)\text{\textbf{b}} = \sum_{i=1}^\infty a_ib_i\psi_i^2(t) \quad,\quad \psi_i^2(t) = 
    \begin{cases}
        2^j &, \displaystyle \text{ if } \frac{k}{2^j} \leq t \leq \frac{k+1}{2^j}; \\[1ex]
        0 &, \text{ otherwise.}
    \end{cases}
\end{equation*}
Then $\psi_i^2 \in L^2([0,1])$, and we can write it as a Haar wavelet series:
\begin{equation*}
    \psi_i^2(t) = \sum_{l=1}^\infty \gamma_{i,l} \psi_l(t) = \Gamma\text{\textbf{h}}(t) \quad,\quad \gamma_{i,l} := \int_0^1 \psi_i^2(t) \psi_l(t) \,dt.
\end{equation*}
Similar to $P$, the matrix $\Gamma$ also has a recursive structure:

\begin{lemma}
    The matrix $\Gamma$ can be recursively calculated by
    \begin{equation}\label{eq:haar_transform_biform_Theta_matrix}
        \Gamma_1 = 1 \quad,\quad \Gamma_{2m} = 
        \begin{bmatrix}
            \Gamma_m & 0_m \\
            H_m^T & 0_m
        \end{bmatrix}
    	\quad, \text{ for } m = 2^j \text{ and }  j = 0,1,2,3,...
    \end{equation}
\end{lemma}

\begin{proof}
    First, calculating $\Gamma_{1} = \gamma_{11}$ is straightforward. By definition, for $i,l = 1,2,...,m$, we have
    \begin{equation*}
        (\Gamma_{2m})_{i,l} = \int_0^1 \psi_i^2(t)\psi_l(t)\,dt = (\Gamma_m)_{i,l}
    \end{equation*}
    
    Suppose now that $m+1 \leq l \leq 2m$, and denote using the two-index notation $\psi_l = \psi_{p,q}$ and $\psi_i = \psi_{j,k}$. For $2 \leq i \leq 2m$, due to the nesting property, $\psi_{j,k}$ is constant (possibly zero) in $\supp\psi_{p,q}$. However, even if $\psi_{j,k}$ is non-zero in $\supp\psi_{p,q}$,
    \begin{equation*}
        (\Gamma_{2m})_{i,l} = \int_0^1 \psi_i^2(t)\psi_l(t)\,dt = 2^j \int_\frac{k}{2^j}^\frac{k+1}{2^j} \psi_{p,q}(t)\,dt = 2^j \int_\frac{q}{2^p}^\frac{q+1}{2^p} \psi_{p,q}(t)\,dt = 0.
    \end{equation*}
    Similar reasoning applies for $i = 1$ (in which case $\psi_1(t) = \phi(t) \equiv 1$ in $[0,1]$). 
    
    Suppose now that $m+1 \leq i \leq 2m$ and $l \leq m$. For $l=1$ a straightforward calculation show that $(\Gamma_{2m})_{i,l} = 1 = (H_m)_{1,i}$. For $l > 1$, since $p < j$, we can apply the same reasoning as in the proof of Lemma~\ref{lem:h_hT_recursive_formula}, yielding
    \begin{equation*}
        \gamma_{i,l} = 2^j \int_\frac{k}{2^j}^\frac{k+1}{2^j} \psi_{p,q}(t)\,dt = 2^j \int_\frac{k}{2^j}^\frac{k+1}{2^j} \psi_{p,q}\left(\frac{k+0.5}{2^j}\right)\,dt = \psi_l(t_i) = (H_m)_{l,i}.
    \end{equation*}
\end{proof}

Hence,
\begin{equation*}
    \text{\textbf{a}}^T\Theta(t)\text{\textbf{b}} = \sum_{i=1}^\infty a_ib_i\psi_i^2(t) =     \sum_{i=1}^\infty a_ib_i \sum_{l=1}^\infty \gamma_{i,l}\psi_l(t) = (\text{\textbf{a}}\odot\text{\textbf{b}})^T \Gamma \, \text{\textbf{h}}(t)
\end{equation*}
Thus, if $\Gamma^T (\text{\textbf{a}}\odot\text{\textbf{b}}) \in \ell^2(\R)$, then $\mathcal{H}(\text{\textbf{a}}^T\Theta(t)\text{\textbf{b}}) = \Gamma^T (\text{\textbf{a}}\odot\text{\textbf{b}})$.

Now we must prove that our tentative Haar transforms are indeed elements of $\ell^2(\R)$; Theorem \ref{thm:biform_haar_transform} follows then from the uniqueness of the Haar series. Before that, we prove a few lemmas:

\begin{lemma}\label{lem:P_Omega_structure}
    The matrix $P\Tilde{\Omega}$ is recursively given by
    \begin{equation}\label{eq:P_Omega_recursive_formula}
        P_{2m}\Tilde{\Omega}_{2m} = 
        \begin{bmatrix}
            P_m\Tilde{\Omega}_m & P_mH_m \\[10pt]
            \displaystyle\frac{1}{4\sqrt{m^3}}H_m^T\Tilde{\Omega}_m & \displaystyle\frac{1}{4\sqrt{m}}I_m
        \end{bmatrix}.
    \end{equation}
\end{lemma}

This is proven by multiplying the recursive formulas for $P$ and $\Tilde{\Omega}$.

\begin{lemma}\label{lem:H_Omega_structure}
    \label{prop:HmOmegam_element}
    The matrix $H_m^T\Tilde{\Omega}_m$ is given element-wise as
    \begin{equation}
        (H_m^T\Tilde{\Omega}_m)_{i,l} = 
        \begin{cases}
            0 & \text{, if } l = 1; \\
            \psi_l^2(t_i) & \text{, otherwise.}
        \end{cases}
    \end{equation}
\end{lemma}

\begin{proof}
First, observe that, since the first column of the matrix $\Tilde{\Omega}_m$ is zero, then the first column of $H_m^T\Tilde{\Omega}_m$ is also zero, proving the case $l = 1$. 

For $l \geq 2$, fix an element $(H_m^T\Tilde{\Omega}_m)_{i,l}$ and make $\psi_i = \psi_{j,k}$ and $\psi_l = \psi_{p,q}$ using the two-index notation. Due to the structure of $H_m^T$ and $\Tilde{\Omega}_m$,
\begin{align*}
    (H_m^T\Tilde{\Omega}_m)_{i,l} = \text{\textbf{h}}_{2^p}^T(t_i) \text{\textbf{h}}_{2^p}(t_{q+1}) = 1 + \sum_{r=0}^{p-1}\sum_{s=0}^{2^r-1} \psi_{r,s}(t_i)\psi_{r,s}(t_{q+1}).
\end{align*}
where $t_i = 2^{-j}(k-0.5)$ and $\displaystyle t_{q+1} = 2^{-p}(q+0.5)$. Since $t_i$ and $t_{q+1}$ are fixed, for each $r$ there is at most a single wavelet $\psi_{r,s_r}$ whose support contains both $t_i$ and $t_{q+1}$, because the intervals where wavelets at the same resolution level are non-zero do not overlap. Thus,
\begin{equation*}
    (H_m^T\Tilde{\Omega}_m)_{i,l} = 1 + \sum_{r=0}^{p-1} \psi_{r,s_r}(t_i)\psi_{r,s_r}(t_{q+1})
\end{equation*}
If neither $\psi_{r,s_r}(t_i)$ nor $\psi_{r,s_r}(t_{q+1})$ are zero, only two cases may occur: either $\psi_{r,s_r}(t_i) = \psi_{r,s_r}(t_{q+1})$ or $\psi_{r,s_r}(t_i) = -\psi_{r,s_r}(t_{q+1})$. The possible situations are depicted in Figure \ref{fig:product_psi_rs}, supposing without loss of generality that $t_i \leq t_{q+1}$.

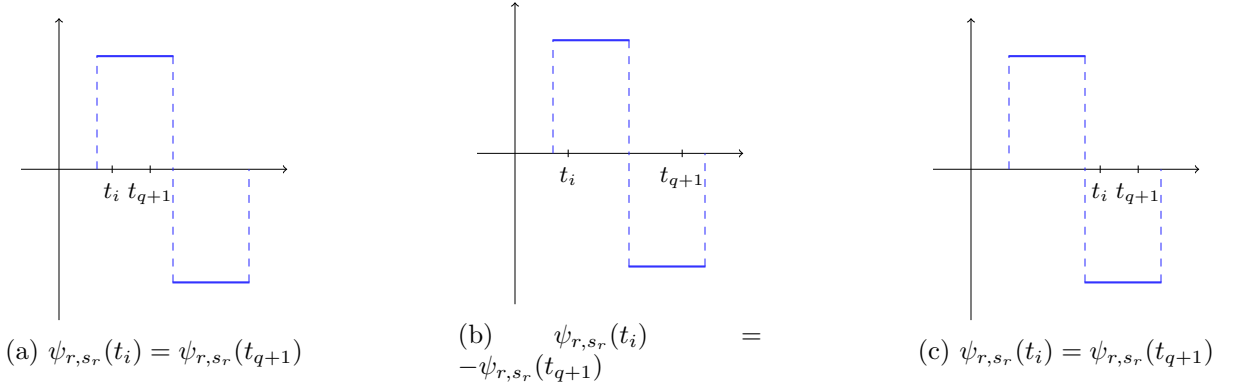
\begin{figure}[ht]
    \centering
    \begin{subfigure}{0.25\textwidth}
        \centering
        \begin{tikzpicture}
            \draw[->,very thin] (0,-2) -- (0,2);
            \draw[->,very thin] (-0.5,0) -- (3,0);
            
            \draw[dashed,blue!80] (0.5,0) -- (0.5,1.5);
            \draw[thick,blue!80] (0.493,1.5) -- (1.507,1.5);
            \draw[dashed,blue!80] (1.5,1.5) -- (1.5,0);
            \draw[dashed,blue!80] (1.5,-1.5) -- (1.5,0);
            \draw[thick,blue!80] (1.493,-1.5) -- (2.507,-1.5);
            \draw[dashed,blue!80] (2.5,-1.5) -- (2.5,0);
            
            \draw[thin] (0.7,0.05) -- (0.7,-0.05) node[anchor=north]{\footnotesize $t_i$};
            \draw[thin] (1.2,0.05) -- (1.2,-0.05) node[anchor=north]{\footnotesize $t_{q+1}$};
        \end{tikzpicture}
        \caption{$\psi_{r,s_r}(t_i) = \psi_{r,s_r}(t_{q+1})$}
    \end{subfigure}
    \hfill
    \begin{subfigure}{0.25\textwidth}
        \centering
        \begin{tikzpicture}
            \draw[->,very thin] (0,-2) -- (0,2);
            \draw[->,very thin] (-0.5,0) -- (3,0);
            
            \draw[dashed,blue!80] (0.5,0) -- (0.5,1.5);
            \draw[thick,blue!80] (0.493,1.5) -- (1.507,1.5);
            \draw[dashed,blue!80] (1.5,1.5) -- (1.5,0);
            \draw[dashed,blue!80] (1.5,-1.5) -- (1.5,0);
            \draw[thick,blue!80] (1.493,-1.5) -- (2.507,-1.5);
            \draw[dashed,blue!80] (2.5,-1.5) -- (2.5,0);
            
            \draw[thin] (0.7,0.05) -- (0.7,-0.05) node[anchor=north]{\footnotesize $t_i$};
            \draw[thin] (2.2,0.05) -- (2.2,-0.05) node[anchor=north]{\footnotesize $t_{q+1}$};
        \end{tikzpicture}
        \caption{$\psi_{r,s_r}(t_i) = -\psi_{r,s_r}(t_{q+1})$}
    \end{subfigure}
    \hfill
    \begin{subfigure}{0.25\textwidth}
        \centering
        \begin{tikzpicture}
            \draw[->,very thin] (0,-2) -- (0,2);
            \draw[->,very thin] (-0.5,0) -- (3,0);
            
            \draw[dashed,blue!80] (0.5,0) -- (0.5,1.5);
            \draw[thick,blue!80] (0.493,1.5) -- (1.507,1.5);
            \draw[dashed,blue!80] (1.5,1.5) -- (1.5,0);
            \draw[dashed,blue!80] (1.5,-1.5) -- (1.5,0);
            \draw[thick,blue!80] (1.493,-1.5) -- (2.507,-1.5);
            \draw[dashed,blue!80] (2.5,-1.5) -- (2.5,0);
            
            \draw[thin] (1.7,0.05) -- (1.7,-0.05) node[anchor=north]{\footnotesize $t_i$};
            \draw[thin] (2.2,0.05) -- (2.2,-0.05) node[anchor=north]{\footnotesize $t_{q+1}$};
        \end{tikzpicture}
        \caption{$\psi_{r,s_r}(t_i) = \psi_{r,s_r}(t_{q+1})$}
    \end{subfigure}
    \caption{All possible situations for the product $\psi_{r,s_r}(t_i)\psi_{r,s_r}(t_{q+1})$}
    \label{fig:product_psi_rs}
\end{figure}

Let us study what happens when we change the resolution level $r$:
\begin{enumerate}[wide,label=\alph*)]
    \item Suppose that $\psi_{r,s_r}(t_i) = \psi_{r,s_r}(t_{q+1})$ (situations (a) or (c) in Figure \ref{fig:product_psi_rs}) for every $r = 0,...,p-1$. Then $t_i = t_{q+1}$, since for $r=p-1$ we only sample $\psi_{r,s_r}$ at the times $t_1=\frac{s_r+0.25}{2^r}$ and $t_2=\frac{s_r+0.75}{2^r}$, and we have $\psi_{r,s_r}(t_1) = -\psi_{r,s_r}(t_2)$. Thus
    \begin{equation*}
        (H_m^T\Tilde{\Omega}_m)_{i,l} = 1 + \sum_{r=0}^{p-1} \psi_{r,s_r}^2(t_i) = 1 + \sum_{r=0}^{p-1} 2^r = 2^p = \psi_l^2(t_i) 
    \end{equation*}
    
    \item Suppose that $\psi_{r,s_r}(t_i) = -\psi_{r,s_r}(t_{q+1})$ (situation (b) in Figure \ref{fig:product_psi_rs}) happens for some $r$ for some $r \leq p-1$. Then $t_i \in [\frac{s}{2^r},\frac{s+0.5}{2^r}]$ and $t_{q+1} \in [\frac{s+0.5}{2^r},\frac{s+1}{2^r}]$. Due to the nesting property of the Haar wavelets, no finer wavelet has both $t_i$ and $t_{q+1}$ in its support, and for all coarser resolutions $\rho < r$ we have $\psi_{\rho,s_\rho}(t_i) = \psi_{\rho,s_\rho}(t_{q+1})$. Hence
    \begin{equation*}
        (H_m^T\Tilde{\Omega}_m)_{i,l} = 1 + \sum_{\rho=0}^{r-1} \psi_{\rho,s_\rho}^2(t_i) -  \psi_{r,s_r}^2(t_i) = 1 + \sum_{\rho=0}^{r-1} 2^\rho - 2^r = 0
    \end{equation*}
    Lastly, since $p > r$ and $t_q \in \text{supp }\psi_{p,q}$, then $t_i \not\in \text{supp}\psi_{p,q}$. Thus $\psi_l^2(t_i) = \psi_{p,q}^2(t_i) = 0 = (H_m^T\Tilde{\Omega}_m)_{i,l}$.
\end{enumerate}
\end{proof}

Now we finally prove that $(\Tilde{\Omega}^T P^T \text{\normalfont{\textbf{c}}}) \odot \text{\normalfont{\textbf{b}}}$ and $\Gamma^T (\text{\normalfont{\textbf{a}}}\odot\text{\normalfont{\textbf{b}}})$ are indeed in $\ell^2(\R)$. 

\begin{proposition}\label{thm:Haar_transf_biform_bounded}
    The sequences $(\Tilde{\Omega}^T P^T \text{\normalfont{\textbf{c}}}) \odot \text{\normalfont{\textbf{b}}}$ and $\Gamma^T (\text{\normalfont{\textbf{a}}}\odot\text{\normalfont{\textbf{b}}})$ are in $\ell^2(\R)$ and satisfy, for some $C_1,C_2 > 0$,
    \begin{align*}
        \big\| (\Tilde{\Omega}^T P^T \text{\normalfont{\textbf{c}}}) \odot \text{\normalfont{\textbf{b}}} \big\|_{\ell^2} &\leq C_1 \|\text{\normalfont{\textbf{c}}}\|_{\ell^2} \|\text{\normalfont{\textbf{d}}}\|_{\ell^2} \\
        \left\| \Gamma^T (\text{\normalfont{\textbf{a}}}\odot\text{\normalfont{\textbf{b}}}) \right\|_{\ell^2(\R)} &\leq C_2 \|\text{\normalfont{\textbf{c}}}\|_{\ell^2} \|\text{\normalfont{\textbf{d}}}\|_{\ell^2}
    \end{align*}
\end{proposition}

\begin{proof}
    For $(\Tilde{\Omega}^T P^T \text{\normalfont{\textbf{c}}}) \odot \text{\normalfont{\textbf{b}}}$, if we divide $\text{\normalfont{\textbf{c}}}$ as in \eqref{eq:vec_block_division} and using the block structure from Lemma \ref{lem:P_Omega_structure}, each block of $(\Tilde{\Omega}^T P^T \text{\normalfont{\textbf{c}}})$ is given by
    \begin{align*}
        (\Tilde{\Omega}^T P^T \text{\normalfont{\textbf{c}}})_1 &= \sum_{r=0}^\infty \frac{1}{2^{\frac{3r}{2}+2}} (\text{\normalfont{\textbf{c}}}^*_{2^r})^T H^T_{2^r} \Tilde{\Omega}_{2^r} \\
        (\Tilde{\Omega}^T P^T \text{\normalfont{\textbf{c}}})^*_{2^j} &= (c_1,...,c_{2^j}) P_{2^j}H_{2^j} + \frac{1}{2^{\frac{j}{2}+2}} (\text{\normalfont{\textbf{c}}}^*_{2^j})^T + \sum_{r=j+1}^\infty \frac{1}{2^{\frac{3r}{2}+2}} (\text{\normalfont{\textbf{c}}}^*_{2^r})^T H^T_{2^r} \Tilde{\Omega}_{2^r}
    \end{align*}
    We bound each term in the right-hand side:
    \begin{itemize}[wide]
        \item $\|(c_1,...,c_{2^j}) P_{2^j}H_{2^j}\|_{\ell^2} \leq \big\|H_{2^j}^T\big\| \big\|P_{2^j}^T(c_1,...,c_{2^j})\big\|_{\ell^2} \leq \frac{\|\text{\normalfont{\textbf{c}}}\|_{\ell^2}}{2^{\frac{j}{2}+1}}$
        
        \item $\left\| \frac{1}{2^{\frac{j}{2}+2}} (\text{\normalfont{\textbf{c}}}^*_{2^j})^T \right\|_{\ell^2} \leq \frac{\left\|\text{\normalfont{\textbf{c}}}\right\|_{\ell^2}}{2^{\frac{j}{2}+2}}$
        
        \item From Lemma \ref{lem:H_Omega_structure}, $\max H^T_{2^j} \Omega^T_{2^j} = 2^j$; thus
        \begin{equation*}
            \left\|\sum_{r=j+1}^\infty \frac{1}{2^{\frac{3r}{2}+2}} (\text{\normalfont{\textbf{c}}}^*_{2^r})^T H^T_{2^r} \Omega^T_{2^r}\right\|_{\ell^2} \leq \sum_{r=j+1}^\infty \frac{1}{2^{\frac{3r}{2}+2}} \left\|\text{\normalfont{\textbf{c}}}\right\|_{\ell^2} \big\|H^T_{2^r} \Omega^T_{2^r}\big\|_{\ell^2} \leq \frac{1+\sqrt{2}}{2^{\frac{j}{2}+2}} \left\|\text{\normalfont{\textbf{c}}}\right\|_{\ell^2}
        \end{equation*}
    \end{itemize}
    Hence, 
    \begin{equation}\label{eq:OmegaT_PT_c_block_estimate}
        \big\|(\Tilde{\Omega}^T P^T \text{\normalfont{\textbf{c}}})^*_{2^j}\big\|_{\ell^2} \leq \left(1 + \frac{\sqrt{2}}{8}\right) \frac{\left\|\text{\normalfont{\textbf{c}}}\right\|_{\ell^2}}{2^\frac{j}{2}} \quad,\quad 
        |(\Tilde{\Omega}^T P^T \text{\normalfont{\textbf{c}}})_1| \leq  \frac{1+\sqrt{2}}{4} \left\|\text{\normalfont{\textbf{c}}}\right\|_{\ell^2}
    \end{equation}
    Also, for $i = 2^j,2^j+1,...,2^{2j}-1$, $\big|(P^T\text{\normalfont{\textbf{d}}})_i\big| \leq \big\|(P^T \text{\normalfont{\textbf{d}}})^*_{2^j}\big\|_{\ell^2} \leq \frac{\left\|\text{\normalfont{\textbf{d}}}\right\|_{\ell^2}}{2^{j+1}}$, and thus
    \begin{align*}
        \left\| \big( (\Tilde{\Omega}^T P^T \text{\normalfont{\textbf{c}}}) \odot \text{\normalfont{\textbf{b}}} \big)^*_{2^j}\right\|_{\ell^2} &\leq \big\|(\Tilde{\Omega}^T P^T \text{\normalfont{\textbf{c}}})^*_{2^j}\big\|_{\ell^2}\max_{2^j \leq i \leq 2^{2j}-1} \big|(P^T\text{\normalfont{\textbf{d}}})_i\big| \leq \left(1 + \frac{\sqrt{2}}{8}\right) \frac{\left\|\text{\normalfont{\textbf{c}}}\right\|_{\ell^2} \left\|\text{\normalfont{\textbf{d}}}\right\|_{\ell^2}}{2^{\frac{3j}{2}+1}}
    \end{align*}
    Therefore, we can bound $(\Tilde{\Omega}^T P^T \text{\normalfont{\textbf{c}}}) \odot \text{\normalfont{\textbf{b}}}$ by
    \begin{align*}
        \left\| (\Tilde{\Omega}^T P^T \text{\normalfont{\textbf{c}}}) \odot \text{\normalfont{\textbf{b}}} \right\|_{\ell^2}^2 &= 
        \big| \big( (\Tilde{\Omega}^T P^T \text{\normalfont{\textbf{c}}}) \odot \text{\normalfont{\textbf{b}}} \big)_1 \big|^2 + \sum_{j=0}^\infty \left\| \big( (\Tilde{\Omega}^T P^T \text{\normalfont{\textbf{c}}}) \odot \text{\normalfont{\textbf{b}}} \big)^*_{2^j} \right\|_{\ell^2}^2 \\
        &\leq \frac{1}{112}(169+79\sqrt{2}) \left\|\text{\normalfont{\textbf{c}}}\right\|_{\ell^2}^2 \left\|\text{\normalfont{\textbf{d}}}\right\|_{\ell^2}^2 = C_1^2 \left\|\text{\normalfont{\textbf{c}}}\right\|_{\ell^2}^2 \left\|\text{\normalfont{\textbf{d}}}\right\|_{\ell^2}^2.
    \end{align*}
    
    For $\Gamma^T (\text{\normalfont{\textbf{a}}}\odot\text{\normalfont{\textbf{b}}})$, from \eqref{eq:haar_transform_biform_Theta_matrix} we have
    \begin{align*}
        \left|\left( \Gamma^T (\text{\normalfont{\textbf{a}}}\odot\text{\normalfont{\textbf{b}}}) \right)_1\right| &= \left| a_1b_1 + \sum_{r=0}^\infty H_{2^r} \left( \text{\normalfont{\textbf{a}}}^*_{2^r} \odot \text{\normalfont{\textbf{b}}}^*_{2^r} \right) \right| \leq |a_1b_1| + \sum_{r=0}^\infty \|H_{2^r}\|  \left\|\text{\normalfont{\textbf{a}}}^*_{2^r}\right\|_{\ell^2} \max\left|\text{\normalfont{\textbf{b}}}^*_{2^r}\right|  \\
            &\leq \left\|\text{\normalfont{\textbf{c}}}\right\|_{\ell^2} \left\|\text{\normalfont{\textbf{d}}}\right\|_{\ell^2} + \sum_{r=0}^\infty \frac{\left\|\text{\normalfont{\textbf{c}}}\right\|_{\ell^2} \left\|\text{\normalfont{\textbf{d}}}\right\|_{\ell^2}}{2^{\frac{3r}{2}+2}} \leq \frac{18+\sqrt{2}}{14} \left\|\text{\normalfont{\textbf{c}}}\right\|_{\ell^2} \left\|\text{\normalfont{\textbf{d}}}\right\|_{\ell^2} \\
        \left\| \left( \Gamma^T (\text{\normalfont{\textbf{a}}}\odot\text{\normalfont{\textbf{b}}}) \right)^*_{2^j} \right\|_{\ell^2} &= \left\| \sum_{r=j+1}^\infty H_{2^r} \left( \text{\normalfont{\textbf{a}}}^*_{2^r} \odot \text{\normalfont{\textbf{b}}}^*_{2^r} \right) \right\|_{\ell^2} \leq \sum_{r=j+1}^\infty \|H_{2^r}\|  \left\|\text{\normalfont{\textbf{a}}}^*_{2^r}\right\|_{\ell^2} \max\left|\text{\normalfont{\textbf{b}}}^*_{2^r}\right| \\
            &\leq \sum_{r=j+1}^\infty \frac{\left\|\text{\normalfont{\textbf{c}}}\right\|_{\ell^2} \left\|\text{\normalfont{\textbf{d}}}\right\|_{\ell^2}}{2^{\frac{3r}{2}+2}} = \frac{1+2\sqrt{2}}{4} \frac{\left\|\text{\normalfont{\textbf{c}}}\right\|_{\ell^2} \left\|\text{\normalfont{\textbf{d}}}\right\|_{\ell^2}}{2^{\frac{3r}{2}}}
    \end{align*}
    Finally, we bound the sequence $\Gamma^T (\text{\normalfont{\textbf{a}}}\odot\text{\normalfont{\textbf{b}}})$ with
    \begin{align*}
        \left\| \Gamma^T (\text{\normalfont{\textbf{a}}}\odot\text{\normalfont{\textbf{b}}}) \right\|_{\ell^2}^2 &= \left|\left( \Gamma^T (\text{\normalfont{\textbf{a}}}\odot\text{\normalfont{\textbf{b}}}) \right)_1\right|^2 + \sum_{j=0}^\infty \left\| \left( \Gamma^T (\text{\normalfont{\textbf{a}}}\odot\text{\normalfont{\textbf{b}}}) \right)^*_{2^j} \right\|_{\ell^2}^2 \\
        &\leq \frac{1}{49} (113+23\sqrt{2})\left\|\text{\normalfont{\textbf{c}}}\right\|_{\ell^2} \left\|\text{\normalfont{\textbf{d}}}\right\|_{\ell^2} = C_2^2 \left\|\text{\normalfont{\textbf{c}}}\right\|_{\ell^2} \left\|\text{\normalfont{\textbf{d}}}\right\|_{\ell^2}
    \end{align*}
\end{proof}
\subsection{Proofs of quadratic estimates from Section~\ref{sect:nonlinear}}\label{apdx:estimates}

Here we prove the quadratic estimates from Propositions~\ref{prop:estimates_finite_proj_finite_vectors}---\ref{prop:estimates_infinite_proj_infinite_vectors}. As stated in the paper, the main strategy is to employ both the recursive block structures of the matrices from Theorems~\ref{thm:P_recursive_formula} and \ref{thm:biform_haar_transform} and the finite-infinite decomposition from \eqref{eq:finite_infinite_decomposition}. For clarity, Figure~\ref{fig:matrix_structure_overlay} shows how they overlap for the operator $P$; the other matrices follow a similar pattern. We also draw insights from \ref{apdx:biform} to bound the sums that appear in the proof. We believe that similar procedures may be applied for higher-degree polynomial nonlinearities.

\begin{figure}[h!]
    \centering
    \includegraphics[width=0.55\linewidth]{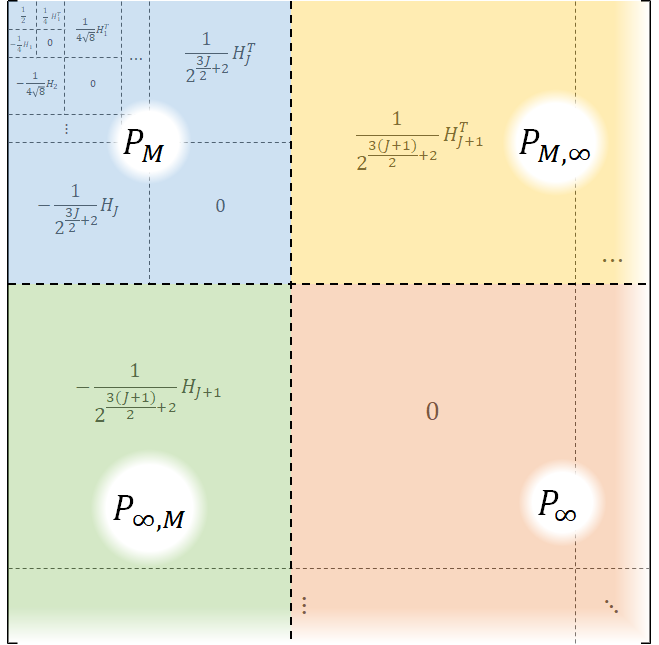}
    \caption{Overlay of the block structure and finite-infinite decomposition for $P$}
    \label{fig:matrix_structure_overlay}
\end{figure}

\subsubsection{Proof of Proposition~\ref{prop:estimates_finite_proj_finite_vectors}}

\begin{enumerate}[wide,label=\textit{\roman*)}]
    \item For $i \leq M$, since $\Bar{\text{\normalfont{\textbf{c}}}}\in\R^M$, the $i$-th element of $\Pi_M P^T \Bar{\text{\normalfont{\textbf{c}}}}$ is given by
    \begin{equation*}
        \left(P^T\Bar{\text{\normalfont{\textbf{c}}}}\right)_i = \sum_{l=1}^M \left(P^T\right)_{i,l}c_l = \left(P_M^T \Bar{\text{\normalfont{\textbf{c}}}}\right)_i.
    \end{equation*}
    
    \item For $i \leq M$,
    \begin{equation*}
        \left(\mathcal{H}(\Bar{\text{\textbf{a}}}^T\Omega(t)\Bar{\text{\textbf{b}}})\right)_i = (\Bar{\text{\normalfont{\textbf{c}}}}^TP\,\Tilde{\Omega})_i
        \Bar{\text{\textbf{b}}}_i = \sum_{l=1}^M c_l(P\,\Tilde{\Omega})_{l,i} \Bar{\text{\textbf{b}}}_i = (\Bar{\text{\normalfont{\textbf{c}}}}^TP_M\Tilde{\Omega}_M)_i \Bar{\text{\textbf{b}}}_i.
    \end{equation*}
    with the last equality due to Lemma~\ref{lem:P_Omega_structure}. Thus, from item \emph{(i)}
    \begin{align*}
        \Pi_M\mathcal{H}(\Bar{\text{\textbf{a}}}^T\Omega(t)\Bar{\text{\textbf{b}}}) &= (\Bar{\text{\normalfont{\textbf{c}}}}^TP_M\Tilde{\Omega}_M)^T \odot \Pi_M\Bar{\text{\textbf{b}}} = (\Tilde{\Omega}_M^TP_M^T\Bar{\text{\normalfont{\textbf{c}}}}) \odot (P_M^T\Bar{\text{\normalfont{\textbf{d}}}}).
    \end{align*}
    
    \item Using the finite-infinite decomposition, we can separate $\Pi_M\mathcal{H}(\Bar{\text{\textbf{a}}}^T\Theta(t)\Bar{\text{\textbf{b}}})$ as
    \begin{equation*}
        \Pi_M\mathcal{H}(\Bar{\text{\textbf{a}}}^T\Theta(t)\Bar{\text{\textbf{b}}}) = \Gamma^T_M (\Bar{\text{\textbf{a}}}_M \odot \Bar{\text{\textbf{b}}}_M) + \Gamma^T_{\infty,M} (\Bar{\text{\textbf{a}}}_\infty \odot \Bar{\text{\textbf{b}}}_\infty)
    \end{equation*}
    We now need to bound the infinite sum in the second term. Adapting the expression of $\Bar{\text{\textbf{b}}} = P^T\Bar{\text{\normalfont{\textbf{d}}}}$ in the proof of Proposition~\ref{prop:PTc_etimates} for $2^j+l > M$ and $\Bar{\text{\normalfont{\textbf{d}}}} \in \R^M$,
    \begin{equation*}
        \Bar{b}_{2^j+l} = -\frac{1}{2^{\frac{3j}{2} + 2}} \sum_{q=1}^M \left(H_{2^j}^T\right)_{l,q}\Bar{c}_q,
    \end{equation*}
    and thus, using the structure of $\Gamma$, \eqref{eq:PTc_block_estimate_finite_vector} and the fact that $H_{2^j}H_{2^j}^T = 2^j I_{2^j}$,
    \begin{align*}
        \left|\left(\Gamma^T_\infty (\Pi_\infty\Bar{\text{\textbf{a}}} \odot \Pi_\infty\Bar{\text{\textbf{b}}})\right)_i\right| &= \left| \sum_{j=J+1}^\infty \sum_{l=1}^{2^j} \left(H_{2^j}\right)_{i,l} \Bar{a}_{2^j+l} \Bar{b}_{2^j+l} \right| \\
        &\hspace{-2em}\leq \sum_{j=J+1}^\infty \left| \sum_{l=1}^{2^j} -\left(H_{2^j}\right)_{i,l} \left(\frac{1}{2^{\frac{3j}{2} + 2}} \sum_{q=1}^M \left(H_{2^j}^T\right)_{l,q}c_q\right) \Bar{b}_{2^j+l} \right| \\
        &\hspace{-2em}\leq \sum_{j=J+1}^\infty \frac{1}{2^{\frac{3j}{2}+2}} \left|\left( \sum_{q=1}^M \sum_{l=1}^{2^j} \left(H_{2^j}\right)_{i,l} \left(H_{2^j}^T\right)_{l,q}c_q\right)\right| \frac{\|\Bar{\text{\normalfont{\textbf{d}}}}\|_{\ell^2}}{2^{j+1}} = \frac{\sqrt{2}\,\|\Bar{\text{\normalfont{\textbf{d}}}}\|_{\ell^2}}{\left(4-\sqrt{2}\right)2^{\frac{3J}{2}+3}} |c_i|
    \end{align*}
    and thus
    \begin{equation*}
        \|\Gamma^T_\infty (\Pi_\infty\Bar{\text{\textbf{a}}} \odot \Pi_\infty\Bar{\text{\textbf{b}}})\|_{\ell^2} \leq \frac{\sqrt{2}\, \|\Bar{\text{\normalfont{\textbf{c}}}}\|_{\ell^2} \|\Bar{\text{\normalfont{\textbf{d}}}}\|_{\ell^2}}{\left(4-\sqrt{2}\right)2^{\frac{3J}{2}+3}}.
    \end{equation*}
\end{enumerate}

\subsubsection{Proof of Proposition~\ref{prop:estimates_infinite_proj_finite_vectors}}

\begin{enumerate}[wide, label=\textit{\roman*)}]    
    \item Using the finite-infinite decomposition, the recursive block structure of $P$ and the fact that $\Bar{\text{\textbf{c}}} \in \R^M$, we have for $j > J$
    \begin{equation}\label{eq:PTc_block_estimate_finite_vector}
        \left\|\left(P^T\Bar{\text{\normalfont{\textbf{c}}}}\right)_{2^j}^* \right\|_{\ell^2} = \left\| -\frac{1}{2^{\frac{3j}{2}+2}} H_{2^j}^T \Bar{\text{\normalfont{\textbf{c}}}} \right\|_{\ell^2} \leq \frac{1}{2^{\frac{3j}{2}+2}} \left\|H_{2^j}^T\right\|\left\|\Bar{\text{\normalfont{\textbf{c}}}}\right\|_{\ell^2}
        = \frac{\left\|\Bar{\text{\normalfont{\textbf{c}}}}\right\|_{\ell^2}}{2^{j+2}}
    \end{equation}
    Estimating as we did in Theorem~\ref{thm:Haar_transf_biform_bounded},
    \begin{equation*}
        \left\|\Pi_\infty\left(P^T\Bar{\text{\normalfont{\textbf{c}}}}\right)\right\|_{\ell^2} = \sqrt{\sum_{j=J+1}^\infty \left\|\left(P^T\Bar{\text{\normalfont{\textbf{c}}}}\right)_{2^j}^*\right\|_{\ell^2}^2} \leq \frac{1}{\sqrt{3}} \frac{\left\|\Bar{\text{\normalfont{\textbf{c}}}}\right\|_{\ell^2}}{2^{J+2}}.
    \end{equation*}
    
    \item For $j > J$, dividing $\Pi_\infty\mathcal{H}(\Bar{\text{\textbf{a}}}^T\Omega(t)\Bar{\text{\textbf{b}}})$ in blocks, 
    \begin{align*}
        \left\|\left( \mathcal{H}(\Bar{\text{\textbf{a}}}^T\Omega(t)\Bar{\text{\textbf{b}}})\right)_{2^j}^* \right\|_{\ell^2} & = \left\| -\frac{1}{2^{\frac{3j}{2}+2}} \left(H_{2^j}^TP_{2^j}^T\Bar{\text{\normalfont{\textbf{c}}}}\right) \odot \left(H_{2^j}^T\Bar{\text{\normalfont{\textbf{d}}}}\right) \right\|_{\ell^2} \leq \frac{1}{2^{\frac{3j}{2}+2}} \left(\max_{1 \leq k \leq 2^j}\left|\left( H_{2^j}^T\Bar{\text{\normalfont{\textbf{c}}}} \right)_k\right|\right) \left\| H_{2^j}^TP_{2^j}^T\Bar{\text{\normalfont{\textbf{d}}}} \right\|_{\ell^2} \\
        &\leq \frac{1}{2^{\frac{3j}{2}+2}} \left\| H_{2^j}^T\Bar{\text{\normalfont{\textbf{c}}}} \right\|_{\ell^2} \left\|H_{2^j}^TP_{2^j}^T\Bar{\text{\normalfont{\textbf{d}}}} \right\|_{\ell^2} \leq  \frac{\left\|\Bar{\text{\normalfont{\textbf{c}}}}\right\|_{\ell^2} \left\|\Bar{\text{\normalfont{\textbf{d}}}}\right\|_{\ell^2}}{2^{j+4}}
    \end{align*}
    and thus
    \begin{align*}
        \left\| \Pi_\infty \left( \mathcal{H}(\Bar{\text{\textbf{a}}}^T\Omega(t)\Bar{\text{\textbf{b}}})\right) \right\|_{\ell^2} &= \sqrt{\sum_{j=J+1}^\infty \left\|\left( \mathcal{H}(\Bar{\text{\textbf{a}}}^T\Omega(t)\Bar{\text{\textbf{b}}})\right)_{2^j}^* \right\|_{\ell^2}^2} \leq \frac{1}{\sqrt{3}} \frac{\left\|\Bar{\text{\normalfont{\textbf{c}}}}\right\|_{\ell^2} \left\|\Bar{\text{\normalfont{\textbf{d}}}}\right\|_{\ell^2}}{2^{2J+4}}.
    \end{align*}
    
    \item To estimate $\Pi_\infty\mathcal{H}(\Bar{\text{\textbf{a}}}^T\Theta(t)\Bar{\text{\textbf{a}}})$, with the block matrix structure from \eqref{eq:haar_transform_biform_Theta_matrix} for $i > M$, 
    \begin{align*}
        \left|\left(\mathcal{H}(\Bar{\text{\textbf{a}}}^T\Theta(t)\Bar{\text{\textbf{b}}})\right)_i\right| &= \left| \Gamma^T_{i,*} (\Bar{\text{\textbf{a}}} \odot \Bar{\text{\textbf{b}}})  \right| = \left| \sum_{p=j+1}^\infty \frac{1}{2^{3p+4}} \sum_{q=1}^{2^p} (H_{2^p})_{i,q} \left((H^T_{2^p})_{q,*}\Bar{\text{\normalfont{\textbf{c}}}}\right) \left((H^T_{2^p})_{q,*}\Bar{\text{\normalfont{\textbf{d}}}}\right) \right| \\
        &\leq \sum_{p=j+1}^\infty \frac{1}{2^{3p+4}} \max_{1 \leq q \leq 2^p}\left| \left((H^T_{2^p})_{q,*}\Bar{\text{\normalfont{\textbf{d}}}}\right) \right| \sum_{q=1}^{2^p} \left|(H_{2^p})_{i,q} \left((H^T_{2^p})_{q,*}\Bar{\text{\normalfont{\textbf{c}}}}\right)\right| \\
        &\leq \sum_{p=j+1}^\infty \left|(I_{2^p})_{i,*}\Bar{\text{\normalfont{\textbf{c}}}}\right| \frac{\|\Bar{\text{\normalfont{\textbf{d}}}}\|_{\ell^2}}{2^{3p+6}}= \frac{|c_i|}{2^{3j+6}} \frac{\|\Bar{\text{\normalfont{\textbf{d}}}}\|_{\ell^2}}{7}
    \end{align*}
    and hence
    \begin{equation*}
        \left\|\Pi_\infty \mathcal{H} (\Bar{\text{\textbf{a}}}^T \Theta(t) \Bar{\text{\textbf{b}}})\right\|_{\ell^2} \leq \frac{1}{21\sqrt{7}} \frac{\|\Bar{\text{\normalfont{\textbf{c}}}}\|_{\ell^2} \|\Bar{\text{\normalfont{\textbf{d}}}}\|_{\ell^2}}{2^{3J+6}}.
    \end{equation*}
\end{enumerate}

\subsubsection{Proof of Proposition~\ref{prop:estimates_finite_proj_infinite_vectors}}

The following Lemma helps estimating terms of the type $A_M(B_M\Bar{\text{\normalfont{\textbf{x}}}} \odot C_M\Bar{\text{\normalfont{\textbf{y}}}})$ for arbitrary $\Bar{\text{\normalfont{\textbf{x}}}},\Bar{\text{\normalfont{\textbf{y}}}} \in \R^M$, which would be tricky otherwise.
\begin{lemma}\label{lem:estimate_product_arbitrary_finite_vectors}
    Given $\Bar{\text{\normalfont{\textbf{x}}}},\Bar{\text{\normalfont{\textbf{y}}}} \in \R^M$ and $M \times M$ matrices $A_M,B_M,C_M$, then
    \begin{equation}\label{eq:estimate_odot_finite}
        \left\|A_M(B_M\Bar{\text{\normalfont{\textbf{x}}}} \odot C_M\Bar{\text{\normalfont{\textbf{y}}}})\right\|_{\ell^2} \leq \left\| A_M\,\text{\normalfont{diag}}(\|(C_M)_{l,*}\|_{\ell^2})B_M \right\| \|\Bar{\text{\normalfont{\textbf{x}}}}\|_{\ell^2} \|\Bar{\text{\normalfont{\textbf{y}}}}\|_{\ell^2},
    \end{equation}
    where
    \begin{equation*}
        \text{\normalfont{diag}}(\Bar{\text{\normalfont{\textbf{x}}}}) := \left[
        \begin{array}{cccc}
            x_1 & \multicolumn{1}{c}{\phantom{00}} &      
                \multicolumn{2}{c}{\multirow{2}{*}{$0$}} \\[0.5ex]
            \multicolumn{1}{c}{\phantom{00}} & x_2 &
                \multicolumn{1}{c}{\phantom{00}} & \multicolumn{1}{c}{\phantom{00}}\\[0.5ex]
            \multicolumn{2}{c}{\multirow{2}{*}{$0$}} & \ddots & \multicolumn{1}{c}{\phantom{00}} \\
            \multicolumn{1}{c}{\phantom{00}} &
                \multicolumn{1}{c}{\phantom{00}} & \multicolumn{1}{c}{\phantom{00}} & x_n 
        \end{array}
        \right] \quad,\quad
        \left(\|(C_M)_{l,*}\|_{\ell^2}\right) = 
        \begin{pmatrix}
            \|(C_M)_{1,*}\|_{\ell^2} \\
            \|(C_M)_{2,*}\|_{\ell^2} \\
            \vdots \\
            \|(C_M)_{M,*}\|_{\ell^2} \\
        \end{pmatrix}.
    \end{equation*}
\end{lemma}
\begin{proof}
    For each element of $A_M(B_M\Bar{\text{\normalfont{\textbf{x}}}} \odot C_M\Bar{\text{\normalfont{\textbf{y}}}})$
    \begin{align*}
        \left|\left(A_M(B_M\Bar{\text{\normalfont{\textbf{x}}}} \odot C_M\Bar{\text{\normalfont{\textbf{y}}}})\right)_i\right| &= \left| \sum_{l=1}^M (A_M)_{i,l}(B_M\Bar{\text{\normalfont{\textbf{x}}}})_l (C_M\Bar{\text{\normalfont{\textbf{y}}}})_l\right| \\
        &\hspace{-4em}\leq \left| \sum_{l=1}^M (A_M)_{i,l}(B_M\Bar{\text{\normalfont{\textbf{x}}}})_l \|(C_M)_{l,*}\|_{\ell^2}\|\Bar{\text{\normalfont{\textbf{y}}}}\|_{\ell^2}\right| = \left| (A_M)_{i,*}\text{diag}\left(\|(C_M)_{l,*}\|_{\ell^2}\right)B_M\Bar{\text{\normalfont{\textbf{x}}}}  \right| \|\Bar{\text{\normalfont{\textbf{y}}}}\|_{\ell^2}.
    \end{align*}
    Taking the $\ell^2$ norm, the result follows.
\end{proof}

\begin{enumerate}[wide,label=\roman*)]
    \item Applying the finite-infinite decomposition to $P$ and observing its block structure as in Figure~\ref{fig:matrix_structure_overlay}, we have that
    \begin{align}\label{eq:PTy_inf_estimate}
        \left\|(P_{\infty,M}^T\,\text{\textbf{y}}_\infty)\right\|_{\ell^2} &= \left\| \sum_{q=J+1}^\infty \frac{1}{2^{\frac{3q}{2}+2}}\left(H_{2^q}\right)_{1:M,*} \text{\textbf{y}}_{2^q}^* \right\|_{\ell^2} \leq \sum_{q=J+1}^\infty \frac{1}{2^{\frac{3q}{2}+2}} \left\|\left(H_{2^q}\right)_{1:M,*}\right\| \left\|\text{\textbf{y}}_{2^q}^*\right\|_{\ell^2}  \notag \\
        &\leq \sum_{q=J+1}^\infty \frac{1}{2^{\frac{3q}{2}+2}} \left\|H_{2^q}\right\| \left\|\text{\textbf{y}}_{2^q}^*\right\|_{\ell^2} \leq \frac{\left\|\text{\textbf{y}}\right\|_{\ell^2}}{2^{J+2}}.
    \end{align}

    \item Fix $j > J$ and make $\text{\textbf{z}}_{2^j} := H_{2^j} \text{\textbf{y}}_{2^j}^*$. From Lemma~\ref{lem:P_Omega_structure},
    \begin{align*}
        \left( \Tilde{\Omega}^TP^T \right)_{M,\infty} \text{\textbf{y}}_\infty &= \sum_{j=J+1}^\infty \frac{1}{2^{\frac{3j}{2}+2}} \left(\Tilde{\Omega}_{2^j}^T H_{2^j}\right)_{1:M,*} \text{\textbf{y}}_{2^j}^* \\
        &= \sum_{j=J+1}^\infty \frac{1}{2^{\frac{3j}{2}+2}} \left(\Tilde{\Omega}_{2^j}^T\right)_{1:M,*} H_{2^j} \text{\textbf{y}}_{2^j}^* = \sum_{j=J+1}^\infty \frac{1}{2^{\frac{3j}{2}+2}} \left(\Tilde{\Omega}_{2^j}^T\right)_{1:M,*} \text{\textbf{z}}_{2^j}
    \end{align*}
    Observe that, for $p = 0,1,...,J$,
    \begin{alignat*}{4}
        \left( \Tilde{\Omega}_{2^j}^T\text{\textbf{z}}_{2^j} \right)_1 &= 0 \quad,\quad 
        \left\| \left( \Tilde{\Omega}_{2^j}^T\text{\textbf{z}}_{2^j} \right)_{2^p}^* \right\|_{\ell^2} &= \left\| H_{2^p}^T (z_1,...,z_{2^p})^T \right\|_{\ell^2}
        \leq 2^\frac{p}{2} \left\| \text{\textbf{z}}_{2^j} \right\|_{\ell^2}
    \end{alignat*}
    and therefore, for each $j > J$,
    \begin{align*}
        \left\| \Tilde{\Omega}_{2^j}^T\text{\textbf{z}}_{2^j} \right\|_{\ell^2} &= \sqrt{ \sum_{p=0}^{j} \left\| \left( \Tilde{\Omega}_{2^j}^T\text{\textbf{z}}_{2^j} \right)_{2^p}^* \right\|_{\ell^2}^2} \leq 2^\frac{j+1}{2} \left\| \text{\textbf{z}}_{2^j} \right\|_{\ell^2} \leq 2^{j+\frac{1}{2}} \left\| \text{\textbf{y}}_{2^j}^* \right\|_{\ell^2}.
    \end{align*}
    Hence, we can bound the norm of $\left(\Tilde{\Omega}^TP^T\right)_{M,\infty}\text{\textbf{y}}_\infty$ with
    \begin{align*}
        \left\| \left(\Tilde{\Omega}^TP^T\right)_{M,\infty}\text{\textbf{y}}_\infty \right\|_{\ell^2} &\leq \sum_{j=J+1}^\infty \frac{1}{2^{\frac{3j}{2}+2}} \left\| \left(\Tilde{\Omega}_{2^j}^T\right)_{1:M,*} \text{\textbf{z}}_{2^j} \right\|_{\ell^2} \\
        &\leq \sum_{j=J+1}^\infty \frac{1}{2^\frac{j+3}{2}} \left\| \text{\textbf{y}}_{2^j}^* \right\|_{\ell^2} \leq \left(1+\sqrt{2}\right) \frac{\left\|\text{\textbf{y}}\right\|_{\ell^2}}{2^\frac{J+3}{2}}
    \end{align*}

    \item From the block structure of \eqref{eq:haar_transform_biform_Theta_matrix}, and using \eqref{eq:PTc_block_estimate_infinite_vector} and \eqref{eq:PTc_block_estimate_finite_vector},
    \begin{align*}
        \left\| \Gamma_{\infty,M}^T(\Bar{\text{\textbf{a}}}_\infty \odot \Pi_\infty P^T\text{\textbf{y}}) \right\|_{\ell^2} &= \left\| \sum_{q=J+1}^\infty  (H_{2^q})_{1:M,*} \left( \Bar{\text{\textbf{a}}}^*_{2^q} \odot (P^T\text{\textbf{y}})^*_{2^q} \right) \right\|_{\ell^2} \\
        &\leq \sum_{q=J+1}^\infty \left\|(H_{2^q})\right\| \left( \max_{1 \leq r \leq 2^q} \left|\left(\Bar{\text{\textbf{a}}}^*_{2^q}\right)_r\right| \right) \left\|(P^T\text{\textbf{y}})^*_{2^q}\right\|_{\ell^2} \\
        &\leq \sum_{q=J+1}^\infty 2^{\frac{q}{2}} \frac{\left\|\Bar{\text{\textbf{c}}}\right\|_{\ell^2}}{2^{q+2}} \frac{\left\|\text{\textbf{y}}\right\|_{\ell^2}}{2^{q+1}} = \frac{\sqrt{2}}{4-\sqrt{2}} \frac{\left\|\Bar{\text{\textbf{c}}}\right\|_{\ell^2} \left\|\text{\textbf{y}}\right\|_{\ell^2}}{2^{\frac{3J}{2}+3}}
    \end{align*}

    \item Applying the finite-infinite decomposition to $A_M\Pi_M\mathcal{H}(\text{\textbf{x}}^T P\,\Omega(t)P^T \text{\textbf{y}})$, we have
    \begin{align*}
        \Pi_M\mathcal{H}(\text{\textbf{x}}^T P\,\Omega(t)P^T \text{\textbf{y}}) &= \big( \text{\textbf{x}}_M^T P_M\Tilde{\Omega}_M \big)^T \odot \big( P_M^T\text{\textbf{y}}_M \big) + \left( \text{\textbf{x}}_M^T P_M\Tilde{\Omega}_M \right)^T \odot \left( P_{\infty,M}^T \text{\textbf{y}}_\infty \right) \\
        &\hspace{2em} + \big( \text{\textbf{x}}_\infty^T (P\,\Tilde{\Omega})_{M,\infty} \big)^T \odot \left( P_M^T\text{\textbf{y}}_M \right) + \big( \text{\textbf{x}}_\infty^T (P\,\Tilde{\Omega})_{M,\infty} \big)^T \odot \left( P_{\infty,M}^T \text{\textbf{y}}_\infty \right).
    \end{align*}
    Using \eqref{eq:estimate_odot_finite}, we can bound the first term with Lemma~\ref{lem:estimate_product_arbitrary_finite_vectors}:
    \begin{align*}
        &\left\|A_M\left( \left(\text{\textbf{x}}_M^T P_M\Tilde{\Omega}_M\right)^T \odot \left(P_M^T\text{\textbf{y}}_M\right) \right) \right\|_{\ell^2} \leq \left\| A_M\, \text{diag}\left(\|(P_M^T)_{i,*}\|_{\ell^2}\right)\, \Tilde{\Omega}_M^TP_M^T \right\| \  \|\text{\textbf{x}}\|_{\ell^2} \|\text{\textbf{y}}\|_{\ell^2}.
    \end{align*}
    For the second term, using the bound from \eqref{eq:PTy_inf_estimate},
    \begin{align*}
        \left|\left( A_M \left( \left( \text{\textbf{x}}_M^T P_M\Tilde{\Omega}_M \right)^T \odot \left( P_{\infty,M}^T \text{\textbf{y}}_\infty \right) \right) \right)_i\right| &= \left|\sum_{l=1}^M\left( A_M \right)_{i,l}\left(\text{\textbf{x}}_M^T P_M\Tilde{\Omega}_M \right)_l \left( P_{\infty,M}^T \text{\textbf{y}}_\infty  \right)_l\right| \\
        &\hspace{-6em}\leq \left|\sum_{l=1}^M\left( A_M \right)_{i,l}\left(\text{\textbf{x}}_M^T P_M\Tilde{\Omega}_M \right)_l \right| \frac{\|\text{\textbf{y}}\|_{\ell^2}}{2^{J+2}} = \left|\left(A_M\Tilde{\Omega}_M^T P_M^T \text{\textbf{x}}_M\right)_i \right| \frac{\|\text{\textbf{y}}\|_{\ell^2}}{2^{J+2}}
    \end{align*}
    and thus
    \begin{equation*}
        \left\|A_M\left( \left( \text{\textbf{x}}_M^T P_M\Tilde{\Omega}_M \right)^T \odot \left( P_{\infty,M}^T \text{\textbf{y}}_\infty \right) \right)\right\|_{\ell^2} \leq  \frac{\left\| A_M\Tilde{\Omega}_M^T P_M^T \right\|}{2^{J+2}} \|\text{\textbf{x}}\|_{\ell^2} \|\text{\textbf{y}}\|_{\ell^2}.
    \end{equation*}
    For the third term, with a similar method to the previous term and the estimate for $\text{\textbf{x}}_\infty^T (P\,\Tilde{\Omega})_{M,\infty}$,
    \begin{align*}
        \left|\left( A_M\left( \left(\text{\textbf{x}}_\infty^T(P\,\Tilde{\Omega})_{M,\infty}\right)^T \odot \left(P_M^T\text{\textbf{y}}_M\right) \right) \right)_i\right| &= \left| \sum_{q=1}^M (A_M)_{i,q} \left(P_M^T\text{\textbf{y}}_M\right)_q \left(\text{\textbf{x}}_\infty^T(P\,\Tilde{\Omega})_{M,\infty}\right)_q \right| \\
        &\hspace{-9em}\leq \left(\max_{1 \leq q \leq M} \left|\left( \text{\textbf{x}}_\infty^T(P\,\Tilde{\Omega})_{M,\infty} \right)_q\right|\right) \left|\sum_{q=1}^M (A_M)_{i,q} \left(P_M^T\text{\textbf{y}}_M\right)_q\right|  \\
        &\hspace{-9em}\leq \left\|\text{\textbf{x}}_\infty^T (P\,\Tilde{\Omega})_{M,\infty}\right\|_{\ell^2} \left|\left( A_MP_M^T\text{\textbf{y}}_M \right)_i\right| \leq \left(1+\sqrt{2}\right) \frac{\left\|\text{\textbf{x}}\right\|_{\ell^2}}{2^\frac{J+3}{2}} \left|\left( A_MP_M^T\text{\textbf{y}}_M \right)_i\right|
    \end{align*}
    and thus
    \begin{align*}
        \left\|A_M \left( \left( \text{\textbf{x}}_\infty^T (P\,\Tilde{\Omega})_{M,\infty} \right)^T \odot \left( P_M^T\text{\textbf{y}}_M \right) \right)\right\|_{\ell^2} &\leq \left(1+\sqrt{2}\right) \frac{\left\|\text{\textbf{x}}\right\|_{\ell^2}}{2^\frac{J+3}{2}} \left\|A_MP_M^T\text{\textbf{y}}_M\right\|_{\ell^2} \\
        &\leq \left(1+\sqrt{2}\right) \left\|A_MP_M^T\right\| \frac{\left\|\text{\textbf{x}}\right\|_{\ell^2} \left\|\text{\textbf{y}}\right\|_{\ell^2}}{2^\frac{J+3}{2}}
    \end{align*}
    For the fourth term, with the estimates for $\text{\textbf{x}}_\infty^T \left((P\,\Tilde{\Omega})_{M,\infty}\right)$ and $P_{\infty,M}^T \text{\textbf{y}}_\infty$, we obtain
    \begin{align*}
        \left\|A_M \left( \left( \text{\textbf{x}}_\infty^T (P\,\Tilde{\Omega})_{M,\infty} \right)^T \odot \left( P_{\infty,M}^T \text{\textbf{y}}_\infty  \right) \right)\right\|_{\ell^2} &= \left\| A_M \, \text{diag}\left(P_{\infty,M}^T \text{\textbf{y}}_\infty\right) \left(\text{\textbf{x}}_\infty^T (P\,\Tilde{\Omega})_{M,\infty}\right)^T \right\|_{\ell^2} \\ 
        &\hspace{-9em}\leq \left\|A_M\right\| \left\|P_{\infty,M}^T \text{\textbf{y}}_\infty \right\|_{\ell^2} \left\|\text{\textbf{x}}_\infty^T(P\,\Tilde{\Omega})_{M,\infty}\right\|_{\ell^2}\leq \left\|A_M\right\| \left(1+\sqrt{2}\right) \frac{\left\|\text{\textbf{x}}\right\|_{\ell^2} \left\|\text{\textbf{y}}\right\|_{\ell^2}}{2^\frac{3J+7}{2}}
    \end{align*}
    Finally, adding the four bounds, we obtain the desired bound.
    
    \item Applying the finite-infinite decomposition to both $P^T$ and $\Gamma^T$, we have
    \begin{align*}
        A_M\Pi_M\mathcal{H}(\text{\textbf{x}}^T P\Theta(t)P^T \text{\textbf{y}}) &= A_M\Gamma_M^T \Big[ \left(P_M^T\text{\textbf{x}}_M\right) \odot \left(P_M^T\text{\textbf{y}}_M\right) + \left(P_M^T\text{\textbf{x}}_M\right) \odot \left(P_{\infty,M}^T\text{\textbf{y}}_\infty\right) \\
        &\hspace{3em} + \left(P_{\infty,M}^T\text{\textbf{x}}_\infty\right) \odot \left(P_M^T\text{\textbf{y}}_M\right) + \left(P_{\infty,M}^T\text{\textbf{x}}_\infty\right) \odot \left(P_{\infty,M}^T \text{\textbf{y}}_\infty\right) \Big] \\
        &\hspace{3em}+ A_M\Gamma_{\infty,M}^T\Pi_\infty \left( P^T \text{\textbf{x}} \odot P^T \text{\textbf{y}} \right).
    \end{align*}
    We estimate it term by term again. For the first one, using Lemma~\ref{lem:estimate_product_arbitrary_finite_vectors},
    \begin{equation*}
        \begin{aligned}
            \left\| A_M\Gamma_M^T \left(\left(P_M^T\text{\textbf{x}}_M\right) \odot \left(P_M^T\text{\textbf{y}}_M\right)\right) \right\|_{\ell^2} &\leq \left\| A_M\Gamma_M^T\text{diag}(\|(P_M^T)_{i,*}\|_{\ell^2})P_M^T \right\| \|\text{\textbf{x}}\|_{\ell^2} \|\text{\textbf{y}}\|_{\ell^2}.
        \end{aligned}
    \end{equation*}
    For the second term, we can use \eqref{eq:PTy_inf_estimate} to bound it element-wise by
    \begin{align*}
        \left| \Big(A_M\Gamma_M^T \left(\left(P_M^T\text{\textbf{x}}_M\right) \odot \left(P_{\infty,M}^T\text{\textbf{y}}_\infty\right) \right) \Big)_i \right| &= \left|\sum_{l=1}^M (A_M\Gamma_M^T)_{i,l} \left(P_M^T\text{\textbf{x}}_M\right)_l \left(P_{\infty,M}^T\text{\textbf{y}}_\infty\right)_l \right| \\
        &\hspace{-6em}\leq \left|\sum_{l=1}^M (A_M\Gamma_M^T)_{i,l} \left(P_M^T\text{\textbf{x}}_M\right)_l \right| \frac{1}{2^{J+2}} \left\|\text{\textbf{y}}\right\|_{\ell^2} = \frac{\left| (A_M \Gamma_M^T P_M^T\text{\textbf{x}}_M)_i \right|}{2^{J+2}} \left\|\text{\textbf{y}}\right\|_{\ell^2}
    \end{align*}
    and thus
    \begin{align*}
        \left\| A_M\Gamma_M^T \left( \left(P_M^T\text{\textbf{x}}_M\right) \odot \left(P_{\infty,M}^T\text{\textbf{y}}_\infty\right) \right)\right\|_{\ell^2} &= \left( \sum_{i=1}^M \left|\Big(A_M\Gamma_M^T \left(\left(P_M^T\text{\textbf{x}}_M\right) \odot \left(P_{\infty,M}^T\text{\textbf{y}}_\infty\right) \right) \Big)_i\right|^2 \right)^\frac{1}{2} \\
        &\hspace{-4em} \leq \left( \sum_{i=1}^M \left| (A_M \Gamma_M^T P_M^T\text{\textbf{x}}_M)_i \right|^2 \right)^\frac{1}{2} \frac{\left\|\text{\textbf{y}}\right\|_{\ell^2}}{2^{J+2}} = \left\| A_M \Gamma_M^T P_M^T\text{\textbf{x}}_M \right\|_{\ell^2} \frac{\left\|\text{\textbf{y}}\right\|_{\ell^2}}{2^{J+2}} \\
        &\hspace{-4em}\leq \frac{\left\| A_M \Gamma_M^T P_M^T\right\|}{2^{J+2}} \left\|\text{\textbf{x}}\right\|_{\ell^2} \left\|\text{\textbf{y}}\right\|_{\ell^2}
    \end{align*}
    The same procedure applied to the third term yields the same bound as above. For the fourth term, using previous estimates,
    \begin{align*}
        \left\| A_M\Gamma_M^T \left( \left(P_{\infty,M}^T\text{\textbf{x}}_\infty\right) \odot  \left(P_{\infty,M}^T\text{\textbf{y}}_\infty\right) \right)\right\|_{\ell^2} &\leq \left\| A_M\Gamma_M^T \right\| \left\|P_{\infty,M}^T\text{\textbf{x}}_\infty\right\|_{\ell^2} \left\|P_{\infty,M}^T\text{\textbf{y}}_\infty\right\|_{\ell^2} \\&\leq \left\|A_M\Gamma_M^T\right\| \frac{\left\|\text{\textbf{x}}\right\|_{\ell^2} \left\|\text{\textbf{y}}\right\|_{\ell^2}}{2^{2J+4}}.
    \end{align*}
    For the fifth term, using the block structure of $\Gamma^T$ in \eqref{eq:haar_transform_biform_Theta_matrix} and previous estimates,
    \begin{align*}
        \left\| A_M\Gamma_{\infty,M}^T\Pi_\infty \left( P^T \text{\textbf{x}} \odot P^T \text{\textbf{y}} \right) \right\|_{\ell^2} &= \left\| A_M \sum_{p=J+1}^\infty (H_{2^p})_{1:M,*} \left( \text{diag}\left(P^T\text{\textbf{x}}\right)_{2^p}^* \right) \left(P^T\text{\textbf{y}}\right)_{2^p}^* \right\|_{\ell^2} \\
        &\leq \|A_M\| \sum_{p=J+1}^\infty \|H_{2^p}\| \left\|\left(P^T\text{\textbf{x}}\right)_{2^p}^*\right\|_{\ell^2} \left\|\left(P^T\text{\textbf{y}}\right)_{2^p}^*\right\|_{\ell^2} \\
        &\leq \|A_M\| \sum_{p=J+1}^\infty 2^\frac{p}{2} \frac{\left\|\text{\textbf{x}}\right\|_{\ell^2}}{2^{p+2}} \frac{\left\|\text{\textbf{y}}\right\|_{\ell^2}}{2^{p+2}} = \frac{\sqrt{2}\,\|A_M\|}{4-\sqrt{2}} \, \frac{\left\|\text{\textbf{x}}\right\|_{\ell^2} \left\|\text{\textbf{y}}\right\|_{\ell^2}}{2^{\frac{3J}{2}+4}}
    \end{align*}
   
    Hence, adding the estimates, the result follows.
\end{enumerate}

\subsubsection{Proof of Proposition~\ref{prop:estimates_infinite_proj_infinite_vectors}}

\begin{enumerate}[wide, label=\roman*)]
    \item To estimate $\mathcal{H}(\Bar{\text{\textbf{c}}}^T\text{\textbf{w}}(t)\text{\textbf{w}}^T\text{\textbf{y}})$, we again decompose in three parts and estimate each one separately:
    \begin{equation*}
        \mathcal{H}(\Bar{\text{\textbf{c}}}^T\text{\textbf{w}}(t)\text{\textbf{w}}^T(t)\text{\textbf{y}}) = \mathcal{H}(\Bar{\text{\textbf{a}}}^T\Omega(t)P^T\text{\textbf{y}}) + \mathcal{H}(\Bar{\text{\textbf{a}}}^T\Omega^T(t)P^T\text{\textbf{y}}) + \mathcal{H}(\Bar{\text{\textbf{a}}}^T\Theta(t)P^T\text{\textbf{y}}).
    \end{equation*}
    
    \begin{itemize}[wide]
        \item $\mathcal{H}(\Bar{\text{\textbf{a}}}^T\Omega(t)P^T\text{\textbf{y}})$: Since $\Bar{\text{\textbf{c}}} \in \R^M$, $\left(\mathcal{H}(\Bar{\text{\textbf{a}}}^T\Omega(t)P^T\text{\textbf{y}})\right)_i = (\Bar{\text{\textbf{c}}}^TP_{2^j}H_{2^j})_i (P^T\text{\textbf{y}})_i$. Therefore, applying the block decomposition to $\mathcal{H}(\Bar{\text{\textbf{a}}}^T\Omega(t)P^T\text{\textbf{y}})$ and some previously calculated estimates from \eqref{eq:PTc_block_estimate_infinite_vector} and \eqref{eq:PTc_block_estimate_finite_vector},
        \begin{align*}
            \left\| \left(\mathcal{H}( \Bar{\text{\textbf{a}}}^T \Omega(t) P^T \text{\textbf{y}}) \right)_{2^j}^* \right\|_{\ell^2} &= \left\| \left(H_{2^j}^T P_{2^j}^T \Bar{\text{\textbf{c}}}\right)_{2^j}^* \odot \left(P^T \text{\textbf{y}}\right)_{2^j}^* \right\|_{\ell^2} \leq \left\| \text{diag} \left((P^T\text{\textbf{y}})_{2^j}^* \right) \right\| \, \left\| \left(H_{2^j}^T P_{2^j}^T \Bar{\text{\textbf{c}}}\right)_{2^j}^* \right\|_{\ell^2} \\
            &\hspace{2em}= \max_{2^j+1 \leq i \leq 2^{j+1}} \left| \left(P^T\text{\textbf{y}} \right)_i \right| \, \left\| \left(H_{2^j}^T P_{2^j}^T \Bar{\text{\textbf{c}}}\right)_{2^j}^* \right\|_{\ell^2} \leq \frac{\|\Bar{\text{\textbf{c}}}\|_{\ell^2} \|\text{\textbf{y}}\|_{\ell^2}}{2^{\frac{3j}{2}+3}}  
        \end{align*}
        and thus
        \begin{equation*}
            \left\| \Pi_\infty\mathcal{H}( \Bar{\text{\textbf{a}}}^T \Omega(t) P^T \text{\textbf{y}}) \right\| = \sqrt{ \sum_{j=J+1}^\infty \left\| \left(\mathcal{H}( \Bar{\text{\textbf{a}}}^T \Omega(t) P^T \text{\textbf{y}}) \right)_{2^j}^* \right\|^2 } \leq \frac{1}{\sqrt{7}} \frac{\|\Bar{\text{\textbf{c}}}\|_{\ell^2} \|\text{\textbf{y}}\|_{\ell^2}}{2^{\frac{3J}{2}+3}}.
        \end{equation*}
        
        \item $\mathcal{H}(\Bar{\text{\textbf{a}}}^T\Omega^T(t)P^T\text{\textbf{y}})$: Since $\Bar{\text{\textbf{a}}}^T\Omega^T(t)P^T\text{\textbf{y}} = (\Bar{\text{\textbf{a}}}^T\Omega^T(t)P^T\text{\textbf{y}})^T = (P^T\text{\textbf{y}})^T\Omega(t)\Bar{\text{\textbf{a}}}$, then using the estimate from \eqref{eq:OmegaT_PT_c_block_estimate},
        \begin{align*}
            \left\|\left( \mathcal{H}(\Bar{\text{\textbf{a}}}^T \Omega^T(t)P^T \text{\textbf{y}}) \right)_{2^j}^*\right\|_{\ell^2} &= \left\| \text{diag}\left((P^T\Bar{\text{\textbf{c}}})_{2^j}^*\right) \big(\Tilde{\Omega}^TP^T\text{\textbf{y}}\big)_{2^j}^* \right\|_{\ell^2} \\
            &\leq \max_{2^j+1 \leq i \leq 2^{j+1}} \left| (P^T\Bar{\text{\textbf{c}}})_i \right| \left(1 + \frac{\sqrt{2}}{8}\right) \frac{\left\|\text{\normalfont{\textbf{y}}}\right\|_{\ell^2}}{2^\frac{j}{2}} \leq \left(1 + \frac{\sqrt{2}}{8}\right) \frac{\|\Bar{\text{\textbf{c}}}\|_{\ell^2} \|\text{\textbf{y}}\|_{\ell^2}}{2^{\frac{3j}{2}+2}},
        \end{align*}
        and the bound for $\Pi_\infty \mathcal{H}(\Bar{\text{\textbf{a}}}^T\Omega^T(t)P^T\text{\textbf{y}})$ becomes
        \begin{equation*}
            \left\| \Pi_\infty\mathcal{H}(\Bar{\text{\textbf{a}}}^T\Omega^T(t)P^T\text{\textbf{y}}) \right\|_{\ell^2} \leq \frac{1}{\sqrt{7}} \left(1 + \frac{\sqrt{2}}{8}\right) \frac{\|\Bar{\text{\textbf{c}}}\|_{\ell^2} \|\text{\textbf{y}}\|_{\ell^2}}{2^{\frac{3J}{2}+2}}.
        \end{equation*}
        
        \item $\mathcal{H}(\Bar{\text{\textbf{a}}}^T\Theta(t)P^T\text{\textbf{y}})$: Using the block structure of $\Gamma^T$ from \eqref{eq:haar_transform_biform_Theta_matrix} and the estimates \eqref{eq:PTc_block_estimate_finite_vector} and \eqref{eq:PTc_block_estimate_infinite_vector}, we have
        \begin{align*}
            \left\| \left( \Gamma^T\left( \Bar{\text{\textbf{a}}} \odot P^T\text{\textbf{y}} \right)\right)_{2^j}^* \right\|_{\ell^2} &= \left\| \sum_{p=j+1}^\infty (H_{2^p})_{2^j+1:2^{j+1},*} \left( \Bar{\text{\textbf{a}}}_{2^p}^* \odot (P^T\text{\textbf{y}})_{2^p}^* \right) \right\|_{\ell^2} \\
            &\leq \sum_{p=j+1}^\infty \left\|(H_{2^p})_{2^j+1:2^{j+1},*}\right\| \left\|\text{diag}(\Bar{\text{\textbf{a}}}_{2^p}^*)\right\| \left\|(P^T\text{\textbf{y}})_{2^p}^*\right\|_{\ell^2} \\
            &\leq  \sum_{p=j+1}^\infty \left\|H_{2^p}\right\| \max_{2^p+1 \leq i \leq 2^{p+1}}\left|(P^T\Bar{\text{\textbf{c}}})_i\right| \left\|(P^T\text{\textbf{y}})_{2^p}^*\right\|_{\ell^2} \\
            &\leq \sum_{p=j+1}^\infty 2^\frac{p}{2} \frac{\left\|\Bar{\text{\textbf{c}}}\right\|_{\ell^2}}{2^{p+2}} \frac{\left\|\text{\textbf{y}}\right\|_{\ell^2}}{2^{p+1}} = \frac{\sqrt{2}}{4-\sqrt{2}} \frac{\left\|\Bar{\text{\textbf{c}}}\right\|_{\ell^2} \left\|\text{\textbf{y}}\right\|_{\ell^2}}{2^{\frac{3j}{2}+3}},
        \end{align*}
        and thus
        \begin{align*}
            \left\|\Pi_\infty\mathcal{H} (\Bar{\text{\textbf{a}}}^T\Theta(t)P^T\text{\textbf{y}}) \right\|_{\ell^2} \leq \frac{\sqrt{2}}{\sqrt{7}(4-\sqrt{2})} \frac{\left\|\Bar{\text{\textbf{c}}}\right\|_{\ell^2} \left\|\text{\textbf{y}}\right\|_{\ell^2}}{2^{\frac{3J}{2}+3}}.
        \end{align*}
    \end{itemize}
    
    The result follows from summing the estimates.
    
    \item We proceed in the same way as we did for $\Pi_\infty\mathcal{H} (\Bar{\text{\textbf{c}}}^T\text{\textbf{w}}(t)\text{\textbf{w}}^T(t)\text{\textbf{y}})$, that is, separate it in three parts
    \begin{equation*}
        \mathcal{H}(\text{\textbf{x}}^T\text{\textbf{w}}(t)\text{\textbf{w}}^T(t)\text{\textbf{y}}) = \mathcal{H}(\text{\textbf{x}}^T\Omega(t)P^T\text{\textbf{y}}) + \mathcal{H}(\text{\textbf{x}}^T\Omega^T(t)P^T\text{\textbf{y}}) + \mathcal{H}(\text{\textbf{x}}^T\Theta(t)P^T\text{\textbf{y}})
    \end{equation*}
    and estimate each one.
    
    \begin{itemize}[wide]
        \item $\mathcal{H}(\text{\textbf{x}}^T\Omega(t)P^T\text{\textbf{y}})$:
        Applying the block decomposition to the expression of $\mathcal{H}(\text{\textbf{x}}^T\Omega(t)P^T\text{\textbf{y}})$ and the estimates from \eqref{eq:PTc_block_estimate_infinite_vector} and \eqref{eq:PTc_block_estimate_finite_vector},
        \begin{align*}
            \left\|\left( \mathcal{H}(\text{\textbf{x}}^T\Omega(t)P^T\text{\textbf{y}}) \right)_{2^j}^*\right\|_{\ell^2} &= \left\| (\Tilde{\Omega}^TP^T\text{\textbf{x}})_{2^j}^* \odot (P^T\text{\textbf{y}})_{2^j}^* \right\|_{\ell^2} \leq \left\|\text{diag}\left(P^T\text{\textbf{y}}\right)_{2^j}^*\right\| \left\| (\Tilde{\Omega}^TP^T\text{\textbf{x}})_{2^j}^*\right\|_{\ell^2} \\
            &\leq \frac{\|\text{\textbf{y}}\|_{\ell^2}}{2^{j+1}} \left( 3+2\sqrt{2} \right) \frac{\|\text{\textbf{x}}\|_{\ell^2}}{2^{\frac{j}{2}+2}} = \frac{\left( 3+2\sqrt{2} \right)}{2^{\frac{3j}{2}+3}}\|\text{\textbf{x}}\|_{\ell^2} \|\text{\textbf{y}}\|_{\ell^2}
        \end{align*}
        and then
        \begin{equation*}
            \left\| \Pi_\infty\mathcal{H}(\text{\textbf{x}}^T\Omega(t)P^T\text{\textbf{y}}) \right\|_{\ell^2} \leq \frac{\left(3+2\sqrt{2}\right)}{\sqrt{7}} \frac{\|\text{\textbf{x}}\|_{\ell^2} \|\text{\textbf{y}}\|_{\ell^2}}{2^{\frac{3J}{2}+3}}
        \end{equation*}
        
        \item $\mathcal{H}(\text{\textbf{x}}^T\Omega^T(t)P^T\text{\textbf{y}})$: Since $\text{\textbf{x}}^T\Omega^T(t)P^T\text{\textbf{y}} = (\text{\textbf{x}}^T\Omega^T(t)P^T\text{\textbf{y}})^T = \text{\textbf{y}}^T\Omega^T(t)P^T\text{\textbf{x}}$, repeating the same process for the previous item, we also have
        \begin{equation*}
            \left\| \Pi_\infty\mathcal{H}(\text{\textbf{x}}^T\Omega^T(t)P^T\text{\textbf{y}}) \right\|_{\ell^2} \leq \frac{\left(3+2\sqrt{2}\right)}{\sqrt{7}} \frac{\|\text{\textbf{x}}\|_{\ell^2} \|\text{\textbf{y}}\|_{\ell^2}}{2^{\frac{3J}{2}+3}}
        \end{equation*}
        
        \item $\mathcal{H}(\text{\textbf{x}}^T\Theta(t)P^T\text{\textbf{y}})$: Using the block decomposition and the block structure of $\Gamma$, 
        \begin{align*}
            \left\|\left( \Gamma^T\left( P^T\text{\textbf{x}} \odot P^T\text{\textbf{y}} \right)\right)_{2^j}^*\right\|_{\ell^2} & = \left\| \sum_{p=j+1}^\infty (H_{2^p})_{2^j+1:2^{j+1},*} \left( (P^T\text{\textbf{x}})_{2^p}^* \odot (P^T\text{\textbf{y}})_{2^p}^* \right) \right\|_{\ell^2} \\
            &\leq \sum_{p=j+1}^\infty \left\|(H_{2^p})\right\| \, \left\|\text{diag}((P^T\text{\textbf{x}})_{2^p}^*)\right\| \left\|(P^T\text{\textbf{y}})_{2^p}^*\right\|_{\ell^2} \\
            &\leq \sum_{p=j+1}^\infty 2^\frac{p}{2} \frac{\left\|\text{\textbf{x}}\right\|_{\ell^2}}{2^{p+1}} \frac{\left\|\text{\textbf{y}}\right\|_{\ell^2}}{2^{p+1}} = \frac{\sqrt{2}}{4-\sqrt{2}} \frac{\left\|\text{\textbf{x}}\right\|_{\ell^2} \left\|\text{\textbf{y}}\right\|_{\ell^2}}{2^{\frac{3j}{2}+2}}
        \end{align*}
        and thus
        \begin{equation*}
            \left\| \Pi_\infty\mathcal{H}(\text{\textbf{x}}^T\Theta(t)P^T\text{\textbf{y}}) \right\|_{\ell^2} \leq \frac{\sqrt{2}}{\sqrt{7}\left(4-\sqrt{2}\right)}  \frac{\left\|\text{\textbf{x}}\right\|_{\ell^2} \left\|\text{\textbf{y}}\right\|_{\ell^2}}{2^{\frac{3J}{2}+2}} \\
        \end{equation*}
    \end{itemize}
    Putting together the estimates, we have the desired result.
\end{enumerate}

\bibliographystyle{plain}
\bibliography{references}

\end{document}